\documentclass[11pt]{article}
\usepackage{fullpage}
\usepackage{fancyhdr}
\usepackage{amsmath,amssymb,amsthm}
\usepackage{amsmath,amscd}
\usepackage{mathtools}
\usepackage{mathrsfs}
\usepackage{tikz-cd}
\usepackage{authblk}
\usepackage{enumitem}
\usepackage{ stmaryrd }

\usepackage[letterpaper, headsep=1.5cm, margin=1.2in]{geometry}
\usepackage{baskervald}

\usepackage{hyperref}
\usepackage[symbols,nogroupskip,sort=use]{glossaries-extra}

\renewcommand{\contentsname}{Contents\\{\footnotesize\normalfont(A table
		of contents should normally not be included)}}

\newcommand{\Z}{\mathbb{Z}}

\newcommand{\lcm}{\text{lcm}}

\newcommand{\Q}{\mathbb{Q}}

\newcommand{\Spec}{\text{Spec}}

\newcommand{\matcaloe}[1]{\mathcal{O}_{\mathcal{E}}^{ #1}}

\newcommand{\Fisoc}{\mathbf{F}-\mathbf{Isoc}}
\newcommand{\Frac}[1]{\mathbf{Frac}(#1)}

\newcommand{\Mphiok}[2]{\textbf{M}\Phi^\nabla_{ #1, #2}}

\newtheorem{theorem}{Theorem}[section]
\newtheorem{lemma}[theorem]{Lemma}

\newtheorem{proposition}[theorem]{Proposition}
\newtheorem{conjecture}[theorem]{Conjecture}
\newtheorem{corollary}[theorem]{Corollary}

\newtheorem{step}{\indent Step}

\theoremstyle{definition}
\newtheorem{definition}[theorem]{Definition}
\newtheorem{remark}[theorem]{Remark}

\newtheorem{example}[theorem]{Example}
\begin{document}
\title{The monodromy of unit-root $F$-isocrystals with geometric origin}

\author{Joe Kramer-Miller}
\maketitle
\begin{abstract}
Let $C$ be a smooth curve over a finite field of characteristic $p$ 
and let $M$
be an overconvergent $\mathbf{F}$-isocrystal over $C$. After replacing $C$ with a 
dense
open subset, $M$ obtains a slope filtration. This is 
a purely $p$-adic phenomenon; there is no counterpart in the theory
of lisse $\ell$-adic sheaves. The graded pieces of this slope
filtration correspond to lisse $p$-adic sheaves, which we call geometric. Geometric 
lisse $p$-adic sheaves are mysterious, as there is no $\ell$-adic analogue. In this 
article 
we study the monodromy of geometric lisse $p$-adic sheaves with rank one. More 
precisely, we prove
exponential bounds on their ramification breaks. When the generic slopes of $M$ are 
integers,
we show that the local ramification breaks satisfy a certain type of periodicity. 
The crux of the proof is
the theory of $\mathbf{F}$-isocrystals with log-decay. We prove a 
monodromy theorem for these $\mathbf{F}$-isocrystals, as well as a theorem
relating the slopes of $M$ to the rate of log-decay of
the slope filtration. 
As a consequence of these methods, we provide a new proof
of the Drinfeld-Kedlaya theorem for irreducible $\mathbf{F}$-isocrystals
on curves.
\end{abstract}

\section{Introduction}
\subsection{Motivation} \label{subsection:motivation}
Let $C$ be a smooth curve over a finite field $\mathfrak{K}=\mathbb{F}_q$ in 
characteristic $p$. 
Classically, the study of motives
over $C$ has focused on lisse $\ell$-adic \'etale sheaves on $C$, where $\ell \neq 
p$. It is 
natural to ask for a $p$-adic
counterpart to the $\ell$-adic theory. However, there are far too many lisse $p$-adic 
\'etale sheaves 
and they
tend to be poorly behaved compared to their $\ell$-adic counterparts. 
For example, if we have a family of ordinary elliptic curves $f:E \to C$, 
the relative first degree $p$-adic \'etale cohomology $R^1_{et} f_* \Q_p$ 
has rank one. In contrast, the relative $\ell$-adic cohomology sheaf 
$R^1_{et}  f_* \Q_\ell$ has rank two, as is expected. Instead, the
correct $p$-adic coefficient objects are overconvergent $\mathbf{F}$-isocrystals,
which were first introduced by Pierre Berthelot (see
\cite{Berthelot1}). 

Overconvergent $\mathbf{F}$-isocrystals have a remarkable extra structure that is 
absent 
in the
$\ell$-adic theory: a \emph{slope filtration}.  Without giving any 
definitions, consider the overconvergent $\mathbf{F}$-isocrystal $M$ that acts
as the $p$-adic counterpart to the lisse sheaf $R^1_{et} f_* \Q_\ell$. The properties
of $M$ follow those of $R^1_{et}  f_* \Q_\ell$. First, $M$
has rank two. Just as in the $\ell$-adic case, for any $x \in C$ we may consider the 
fiber $M_x$ 
and
the action of Frobenius on $M_x$. The characteristic polynomial of this action will 
describe the 
zeta
function of the elliptic curve $E_x$:
\begin{align*}
Z(E_x,s) &= \frac{\det (1 - \text{Frob}^*s, M_x)}{(1-s)(1-q^{deg(x)}s)}.
\end{align*}
Here, we see a fundamental difference between the $\ell$-adic and $p$-adic 
situations. The 
roots
of the numerator of $Z(E_x,s)$ are both $\ell$-adic units. However, since
$E_x$ is ordinary, one root is a $p$-adic unit
and the other root has $q$-adic valuation one. Even before the modern definition of 
an 
$\mathbf{F}$-isocrystal
was in place, Dwork discovered something miraculous with no $\ell$-adic analogue: 
these unit 
roots
come from a rank one subobject $M^{unit}$ of $M$ existing in a larger category of 
convergent
$\mathbf{F}$-isocrystals. It was later demonstrated by Katz in 
\cite{Katz-p-adic_properties} that any ``unit-root'' $\mathbf{F}$-isocrystal 
corresponds to a $p$-adic \'etale sheaf. 
As one may expect, the $p$-adic \'etale sheaf corresponding
to $M^{unit}$ is $R^1_{et}  f_* \Q_p$.

This phenomenon generalizes. Let $N$ be an overconvergent 
$\mathbf{F}$-isocrystal
on $C$ and assume that the Newton polygon of $\det(1 - \text{Frob}^* s, N_x)$ remains
constant as we vary $x \in C$. Katz proves in \cite{katz-slope_filtration} that $N$ 
obtains 
an increasing filtration
in the larger category of convergent $\mathbf{F}$-isocrystals. The graded pieces of 
this filtration are ``twists''
of unit-root $\mathbf{F}$-isocrystals and thus correspond to lisse $p$-adic \'etale 
sheaves on $C$.
We say that a lisse $p$-adic \'etale sheaf is \emph{geometric} if it arises in 
this manner. Geometric $p$-adic \'etale sheaves remain mysterious. When one 
studies properties of overconvergent $\mathbf{F}$-isocrystals,
such as their cohomology or Frobenius distributions, the $\ell$-adic theory often 
serves as a guiding light
suggesting what is true and occasionally how it should be proven.
However, as there is no $\ell$-adic analogue to the slope filtration, it is less 
clear how
to proceed in developing a coherent theory. It is natural to ask if all geometric 
$p$-adic 
\'etale sheaves
share certain properties. Or, more ambitiously, is it possible to determine when a 
$p$-adic 
\'etale sheaf
is geometric? 

In this article we study the monodromy of geometric $p$-adic \'etale sheaves of 
rank one and the ``growth'' properties
of the slope filtration. In the case where the $\mathbf{F}$-isocrystal has 
integral slopes, we prove a monodromy stability result for geometric $p$-adic 
\'etale sheaves. This result says that the ramification breaks satisfy a 
certain type of periodicity. We are naturally led to consider 
$\mathbf{F}$-isocrystals with logarithmic decay and we prove a 
monodromy theorem for these $\mathbf{F}$-isocrystals. We 
also establish a relationship between the Frobenius slopes and the rate of 
logarithmic decay of the slope filtration. This allows us to give a new proof of 
the Drinfeld-Kedlaya theorem.

\subsection{Monodromy results} \label{intro:monodromy results section}

\subsubsection{Local results}\label{intro: local results}

Let $F$ be either $\mathfrak{K}((T))$ or a finite extension of $\Q_p$ and let $G_F$ 
be the 
absolute Galois group of $F$.
We let $L$ be a finite extension of $\Q_p$ with ring of
integers $\mathcal{O}_L$. For $k\geq 0$, we set $O_k^\times = \{x \in 
\mathcal{O}_L^\times 
~~|~~v_p(1-x)> k\}$.
Consider a continuous character $\rho:G_F \to \mathcal{O}_L^\times$.
We define $s_k(\rho)$ to be the largest upper numbering 
ramification
break of the Galois extension of $F$ corresponding to $\rho^{-1}(O_k^\times)$. 
When $F$ has characteristic $0$,
a celebrated result of Sen (see \cite{Sen}) tells us that there exists a positive rational number $c$ such 
that $ke -c \leq s_k(\rho) \leq ke + c$, where $e$ is the ramification index of 
$F$ over $\Q_p$.  Sen's theorem fails dismally in equal characteristic, since 
$s_k(\rho)$ may
grow arbitrarily fast with respect to $k$. In this article we study the growth of 
$s_k(\rho)$ when $F=\mathfrak{K}((T))$ and $\rho$ has geometric origin. We show that 
$s_k(\rho)$ grows exponentially
and under some additional geometric assumptions, we show that $s_k(\rho)$
satisfies a certain periodicity.

\begin{definition}
	We say that $\rho$ has \emph{finite monodromy} if the image
	of the inertia subgroup of $G_F$ is finite. 
	For $r> 0$, 
	we say that $\rho$ has $r$-\emph{bounded monodromy} if there exists
	a positive rational number $c$ such that
	\begin{align*}
	s_k(\rho)&<cp^{rk},
	\end{align*}
	for all $k>0$ (note that when $\rho$ has finite monodromy there exists
	$c>0$ such that $s_k(\rho)<c$ for all $k$, and thus $\rho$ has $r$-bounded monodromy
	for all $r>0$). 
	Let $a\in \Z_{\geq 1}$ and let $s=v_p(q^{a})$, where $v_p$ denotes
	the $p$-adic valuation normalized so that $v_p(p)=1$.
	We say that $\rho$
	has $a$-\emph{stable monodromy} if for every $k \in [0,s]$, 
	there exists $m_k$ and $b_k$ such that
	\begin{align*}
	s_{k+sn}(\rho)=m_kq^{an} + b_k
	\end{align*}
	for $n\gg0$.
	We say that $\rho$ has \emph{stable monodromy} if it has $a$-stable monodromy
	for some $a$.
\end{definition}

We now restrict ourselves to the case where $F=\mathfrak{K}((T))$. 
Let us explain what it means for a character of $G_F$
to be geometric. 
Let $M$ be an overconvergent $\mathbf{F}$-isocrystal over $\Spec(F)$ with 
coefficients in $L$
and let $\iota^{\dagger}(M)$ be the corresponding convergent 
$\mathbf{F}$-isocrystal (see \S \ref{subsection: fisocrystal section}). Then 
$\iota^\dagger(M)$
has a Frobenius slope filtration (see \S \ref{subsection: slope filtrations}):
\begin{align*}
0=M_0\subset M_1 \subset \dots \subset M_d = \iota^\dagger(M),
\end{align*}
where $gr_i(M)=M_i/M_{i-1}$ is isoclinic of slope $\alpha_i$. After
enlarging $L$, we may associate to $\det(M_i)$
a character $\rho_i:G_F \to \mathcal{O}_L^\times$. This character
is well-defined up to twist by an unramified character. 
We say that a character of $G_F$ is \emph{geometric} if it arises this way.

\begin{theorem} \label{intro: main local result2}
	Let $r_i=\frac{1}{\alpha_{i+1}-\alpha_i}$.
	\begin{enumerate}
		\item Then $\rho_i$ has $r_i$-bounded monodromy.
		\item Assume that $M$ is irreducible and that the slopes of $M$ are 
		integers. Then $\rho_i$ has stable monodromy. 
	\end{enumerate}
\end{theorem}

\noindent From Theorem \ref{intro: main local result2}, we see that rank one 
geometric $p$-adic
\'etale sheaves are intricate and fascinating objects. 
This is in stark contrast with the $\ell$-adic
situation, where rank one objects have finite monodromy and are easily 
understood. The stable monodromy 
of $\rho_i$ when $M$ has integral slopes is particularly surprising. Indeed, Theorem 
\ref{intro: main local 
	result2} shows that
the ramification filtration of $\rho_i$ is completely
determined by the first few ramification breaks. This is in contrast
to a general $p$-adic character of $\rho$, where there are essentially no 
restrictions on $s_n(\rho)$.

\begin{example}
	\label{example: geometric}
	Let $f:X \to \Spec(F)$ be a smooth proper morphism and assume
	that $R^i_{et} f_* \Q_p$ is a rank one $p$-adic lisse sheaf. 
	The corresponding character $\rho$ is geometric (see \cite{Kedlaya-thesis}). In 
	particular, 
	Theorem \ref{intro: main local result2} applies
	to many unit-root $\mathbf{F}$-isocrystals studied by Dwork and 
	others:
	\begin{itemize}
		\item Let $M$ be the $\mathbf{F}$-isocrystal associated to an elliptic 
		curve
		$E$ over $\Spec(\mathfrak{K}[[T]])$, whose generic fiber is ordinary and 
		whose 
		special fiber is supersingular. By Theorem \ref{intro: main local 
			result2}, the $p$-adic Tate module of the generic fiber of $E$ has stable 
		monodromy. 
		This was previously
		know by work of Katz-Mazur (see \cite[Chapter 12.9]{Katz-Mazur}). These
		types of ramification bounds for Abelian varieties play a crucial role
		in the theory of $p$-adic modular forms and canonical subgroups.
		\item Let $A \to \Spec(\mathfrak{K}[[T]])$ be a generically ordinary Abelian 
		variety 
		of dimension $2g$ 
		with a non-ordinary special fiber.
		Assume that $A$ has
		multiplication by a real field $K$ of degree $g$ over $\Q$. Then the 
		$\mathbf{F}$-isocrystal
		$M$ associated to $A$ with coefficients in $\Q_p$ has rank $2g$ and has a 
		linear action by $\Q_p\otimes K$. If there is only one prime in $K$ above 
		$p$, so that $\Q_p \otimes K$ is
		a field, we may regard $M$ as an $\mathbf{F}$-isocrystal with coefficients in
		$\Q_p \otimes K$ of rank two. The unit-root subcrystal of the generic fiber 
		of $M$ is 
		a rank 
		one $\mathbf{F}$-isocrystal 
		with 
		coefficients
		in $\Q_p \otimes K$. The corresponding Galois representative is
		surjective by 
		\cite{Ribet-image_of_pi1_hilbert_modular_varieties} and by Theorem 
		\ref{intro: main local result2} it has
		stable monodromy. 
		\item The rank $n+1$ Kloosterman $\mathbf{F}$-isocrystal on 
		$\mathbb{G}_m$ is 
		irreducible
		and ordinary at every point with slopes $\{0,\dots,n\}$, due to work
		of Sperber in \cite{Sperber1}. The unit-root subcrystal has
		stable monodromy at $0$ and $\infty$ by Theorem \ref{intro: 
			main local result2}.
	\end{itemize}
\end{example}

\subsubsection{Global results}

By applying the Riemann-Hurwitz formula together with Theorem
\ref{intro: main local result2},
we may deduce an interesting result about genera growth
along towers of curves. 
Let $U$ be a smooth curve over $\mathfrak{K}$ and let $C$ be its smooth
compactification. Let $\rho:\pi_1^{et}(U) \to \Z_p^\times$
be a continuous representation. For any $k\geq 1$, we let $g_k(\rho)$
denote the genus of the compact curve corresponding to $\rho^{-1}(1+p^k\Z_p) 
\subset \pi_1^{et}(U)$.

\begin{definition}
	Let $a \in \Z_{\geq 1}$ and let $s=v_p(q^a)$. 
	We say that $\rho$ is has $a$-stable genus growth if for every 
	$k \in [0,s]$, there exists a polynomial $b_k(x) \in \Q[x]$ of
	degree $2s$ such that
	\begin{align*}
	g_{k+an}(\rho) &= b_k(q^n),
	\end{align*}
	for $n \gg 0$. We say that $\rho$
	has psuedo-stable genus growth if $\rho$ has $a$-stable genus growth
	for some $a$. 
\end{definition}

\noindent Let $M$ be an overconvergent $\mathbf{F}$-isocrystal 
on $U$  with coefficients in $\Q_p$.
After replacing $U$ with an open dense subset, there exists a slope filtration:
\begin{align*}
0=M_0\subset M_1\subset \dots \subset M_d = \iota^\dagger(M),
\end{align*}
where $gr_i(M)=M_i/M_{i-1}$ is isoclinic of slope $\alpha_i$. As in \S \ref{intro: 
local results}, we may 
associate a character $\rho_i:\pi_1(U) \to \Z_p^\times$ to $\det(M_i)$ (see \S 
\ref{paragraph:galois reps for isoclinic}). 
\begin{theorem} \label{Genus growth with Qp coefficients theorem}
	Assume $M$ is irreducible and has integral slopes. Then for each $i<d$, the 
	representation $\rho_i$ has 
	psuedo-stable genus growth.
\end{theorem}
Let $f:X \to U$ be a smooth proper morphism and let $M$ be the 
overconvergent $\mathbf{F}$-isocrystal $R^i_{cris} f_* \mathcal{O}_{X,cris}$ (see \S \ref{section: berthelot's conjecture} for an explanation of why this is overconvergent). After 
shrinking $U$ we may assume that $M$ has a slope filtration. Thus $\det(R^i_{et} f_* 
\Q_p )$ corresponds to
the character $\rho_i$. Assume that $M$ is irreducible
and has integral slopes. Theorem
\ref{Genus growth with Qp coefficients theorem} implies that
$\rho_i$ has psuedo-stable genus 
growth. This proves a weaker version of a conjecture of Wan, which states that there 
exists
a quadratic $b(x) \in \Q[x]$ such that $s_k(\rho)=b(p^k)$ for large $k$ (see 
\cite[Conjecture 5.2]{Wan-p-rank}).

\subsection{Logarithmic decay and slope filtrations}
\label{intro: log decay slope filtration subsection}
We now give an informal overview of our results on $\mathbf{F}$-isocrystals 
with
logarithmic decay. For simplicity, we restrict ourselves to $\Q_p$-coefficients 
here. See \S \ref{Section: main results} for more general 
statements.
Let $E=\Frac{W(\mathfrak{K})}$, where $W(\mathfrak{K})$ is the $p$-typical Witt 
vectors of $\mathfrak{K}$. We define the integral Amice ring
\[  \mathcal{O}_{\mathcal{E}} := \Bigg\{ \sum_{n=-\infty}^\infty a_nT^n \Bigg |  
\begin{array}  {l}
\text{ We have } a_n\in W(\mathfrak{K}) \text{ and } ~\lim\limits_{n\to-\infty} 
v_p(a_n)=\infty.   
\end{array}
\Bigg \}.
\]
We let $\mathcal{E}$ be the field of fractions of $\mathcal{O}_\mathcal{E}$ and  let 
$\mathcal{E}^\dagger$ be the subring 
of $\mathcal{E}$ consisting of Laurent series convergent on some annulus $r < 
|T|_p < 1$.
Let $\sigma:\mathcal{O}_\mathcal{E} \to \mathcal{O}_\mathcal{E}$ 
act on $E$ as the $p$-Frobenius map and send $T \mapsto T^p$.
A convergent $\mathbf{F}$-isocrystal over $\Spec(F)$ is
a finite dimensional vector space over $\mathcal{E}$
with an isomorphism $\varphi: M \otimes_{\sigma}\mathcal{E} \to M$
and a compatible differential equation (see \S \ref{section: Local F-isocrystals 
	and their slope filtrations}). An overconvergent $\mathbf{F}$-isocrystal is a 
	finite 
dimensional 
vector space over $\mathcal{E}^\dagger$ with the same extra structure.

Given an overconvergent $\mathbf{F}$-isocrystal $M$,
the convergent $\mathbf{F}$-isocrystal 
$M\otimes_{\mathcal{E}^\dagger} \mathcal{E}$ has a slope filtration. In general, the 
steps $M_i$ of the slope filtration
will not be overconvergent. However, it turns out that there
are intermediate ``logarithmic decay'' rings between $\mathcal{E}^\dagger$ and
$\mathcal{E}$, over which $M_i$ are defined. This builds on an idea of Dwork-Sperber 
and 
utilized by 
Wan (see
\cite{Dwork-Sperber} and \cite{WanDworksconjecture}), where
they consider Frobenius structures with logarithmic decay. To define the
$r$-log-decay ring, we
need to introduce naive partial valuations on $\mathcal{E}$.
For any $a(T)=\sum a_nT^n \in \mathcal{E}$ we define
\[w_k(a(T)) = \min_{v_p(a_n)\leq k} \{n\}. \]
That is, $w_k(a(T))$ is the $T$-adic valuation of $a(T)$ reduced modulo 
$p^{k+1}$. We define $\mathcal{E}^r$ to be
the subring of $\mathcal{E}$ consisting of $a(T)$
such that for some $c>0$, we have $w_k(a(T))\geq -cp^{rk} $
for $k$ large.
Roughly, a convergent $\mathbf{F}$-isocrystal has $r$-log-decay
if the Frobenius and differential equation descend to $\mathcal{E}^r$ (the actual 
definition we use is a bit more 
subtle, see \S \ref{paragraph: the log decay cond}). The following theorem states 
that the rate of 
logarithmic decay is closely related to
differences between consecutive slopes. We further conjecture 
that if $M$ is irreducible, then the slopes are entirely determined by
the rate of log-decay (see Conjecture \ref{log-decay and slope filtrations}).
\begin{theorem} \label{intro: log decay slope fil}
	Let $r=\frac{1}{\alpha_{i+1}-\alpha_i}$. Then
	$M_i$ has $r$-log-decay.
\end{theorem}

\noindent We also study the monodromy of rank one $\mathbf{F}$-isocrystals with 
logarithmic 
decay.
\begin{proposition} \label{intro: no log-decay for r<1}
	Let $N$ be a rank one convergent $\mathbf{F}$-isocrystal with $r$-log-decay. 
	Let $\rho: G_F \to \Z_p^\times$ be the corresponding character (well-defined 
	up to unramified twist). Then $\rho$ has $r$-bounded monodromy.
\end{proposition}
\noindent 	
The first part of Theorem \ref{intro: main local result2} follows from Theorem 
\ref{intro: log decay slope fil} and Proposition \ref{intro: no log-decay for 
	r<1}. There is another somewhat surprising consequence of Theorem \ref{intro: log 
	decay 
	slope fil} and Proposition \ref{intro: no log-decay for r<1}: a
proof of the Drinfeld-Kedlaya
theorem for irreducible $\mathbf{F}$-isocrystals on curves. This result first 
appears in
(see \cite{Drinfeld-Kedlaya} and \cite[Appendix A]{Kedlaya-notes_on_isocrystals}), 
though a 
local version appeared in Kedlaya's thesis (see \cite{Kedlaya-thesis}). See
Remark \ref{remark: drinfeld-kedlaya} for a comparison of
our approach to the work of Kedlaya.
\begin{corollary} (Drinfeld-Kedlaya) \label{drinfeld-kedlaya}
	Let $M$ be an irreducible $\mathbf{F}$-isocrystal on a smooth curve $U$
	and let $\alpha_1<\dots<\alpha_d$ be the generic slopes of $M$.
	Then $|\alpha_{i+1}-\alpha_i|\leq 1$.
\end{corollary}

\begin{remark}
	It is possible to prove the Drinfeld-Kedlaya theorem using logarithmic decay 
	without studying representations with infinite monodromy. This is the content 
	of work of the author (see \cite{JKM-higher_dim}), where 
	the Drinfeld-Kedlaya theorem follows from studying connections with 
	logarithmic decay. It was discovered somewhat accidentally that one could 
	deduce the Drinfeld-Kedlaya theorem by studying ramification filtrations, and 
	we view it as a fortunate consequence of the main results of this paper. 
\end{remark}

\subsection{The question of more general ground fields}
It would be interesting to see if Theorem \ref{intro: main local result2}
or Theorem \ref{Genus growth with Qp coefficients theorem} still hold if
we only require that $\mathfrak{K}$ have characteristic $p$. 
In this case, defining the ramificaiton filtration is more nuanced (see work
of Abbes-Saito \cite{abbes-saito-ramification_theory}). However,
in the case of $\mathbf{F}$-isocrystals with finite monodromy,
the differential structure determines the ramification invariants
through a differential Swan conductor defined by Kedlaya (see \cite{kedlaya2007swan}).
This is due to work of Chiarellotto-Pulita and Xiao (see 
\cite{pulita-chiarellotto-differen_swan_conductors} and 
\cite{xiao-ramification_equal_char}). This suggests that
monodromy stability may hold for more general $\mathfrak{K}$.

\subsection{Outline}
In \S \ref{section: some rings} we introduce several rings that will be used 
throughout
the article. In \S \ref{section: Local F-isocrystals and their slope filtrations} we
give an overview of $\mathbf{F}$-isocrystals and introduce the notion of logarithmic 
decay.
Next, in \S \ref{section: ramification} we discuss ramification theory for $p$-adic 
characters
and in \S \ref{section: Monodromy section} we prove a monodromy theorem for rank one
$\mathbf{F}$-isocrystals. We prove several results on recursive Frobenius equations 
in \S \ref{section: recursive frobenius} and we study the growth of the slope 
filtration in \S \ref{section: solving the unit-root subspace}. Finally, in \S 
\ref{Section: main results} we combine the results of \S \ref{section: Monodromy 
section}-\ref{section: solving the unit-root subspace} to prove our main results. 

\subsection{Acknowledgments}

We thank Daqing Wan for his encouragement and enthusiasm for 
this work, as well as for several insightful discussions. It should
be noted that his conjecture was a major motivation for us pursuing this work. We 
also 
acknowledge Kiran Kedlaya, with whom
we had several discussions about the Drinfeld-Kedlaya theorem and slope filtrations.
We have had several useful discussions with Raju Krishnamoorthy.
Finally, we would like to thank Liang Xiao for some helpful comments on an earlier 
version
of this manuscript and we would like to thank an anonymous referee for many helpful 
suggestions.

\section{Conventions}
\label{section: conventions}
Let $R$ be any ring. 
If $R$ has a valuation $v:R \to \mathbb{R}$ and $x \in R$ satisfies
$v(x)>0$, we let $v_x(\cdot)$ denote the normalization
of $v$ satisfying $v_x(x)=1$. Let $\mathcal{O}_R \subset R$
be the subring of elements with $v(x)\geq 0$. For a matrix $A \in M_{n\times m}(R)$,
we let $v(A)$ denote the infimum of the valuations of its entries.  
For $f \in \mathcal{O}_R$,
we let $\overline{f}$ denote the image of $f$ in the residue field.
For any field $F$ we let $G_{F}$ denote the absolute Galois group of $F$.
If $L$ is a Galois extension of $F$ we let $G_{L/F}$ denote the Galois group of
$L$ over $F$. We will assume all characters/representations of $G_F$ are continuous.

The following conventions will be used throughout the article. Let $p$ be prime
and let $q=p^f$.
We will always take $L$ to be a finite extension of $\Q_p$ with residue field 
$\mathbb{F}_q$. Let $e$ be the ramification index of $L$ over $\Q_p$ and let $\pi$ be 
a uniformizing 
element of $L$. We let $\Theta$ denote the set of 
embeddings of $L$ into $\Q_p^{alg}$. Let 
$\mathfrak{K}$ be a perfect field containing $\mathbb{F}_{q}$
and let $E$ be
$\Frac{\mathcal{O}_L\otimes_{W(\mathbb{F}_{q})}W(\mathfrak{K})}$, where $W(R)$ denotes
the ring of $p$-typical Witt vectors of $R$. Then $E$ is totally ramified over
$\Frac{W(\mathfrak{K})}$. Let $\nu$ denote the $p$-Frobenius 
endomorphism on $\Frac{W(\mathfrak{K})}$ and let $\sigma$ denote the endomorphism $1 
\otimes 
\nu^f$ on $E$. 

\section{Some rings}
\label{section: some rings}
\subsection{The Amice ring and the bounded Robba ring}
\label{subsection: The Amice ring and its subrings}
Let $F=\mathfrak{K}((T))$ and let $F^{un}= \mathfrak{K}^{alg}((T))$. We define
the following $E$-algebras:
\[  \mathcal{E}_{F,L} := \Bigg\{ \sum_{n=-\infty}^\infty a_nT^n \Bigg |  
\begin{array}  {l}
\text{ We have } a_n\in E, ~\lim\limits_{n\to-\infty} v_p(a_n)=\infty,   \\
\text{ and the}	~	v_p(a_n) \text{ is bounded below.} 
\end{array}
\Bigg \},  
\]  
\[
\mathcal{E}_{F,L}^\dagger := \Bigg\{ \sum_{n=-\infty}^\infty a_nT^n  \in 
\mathcal{E}_{F,L} \Big |  
\begin{array}  {l}
\text{ There exists $m>0$ such that} \\
v_p(a_n) \geq -mn \text{ for $n\ll 0$} 
\end{array}
\Bigg \}.  
\]
If there is no ambiguity, we will omit the $F$ and the $L$.
Note that $\mathcal{E}^\dagger$ and 
$\mathcal{E}$ are local fields with residue field
$F$. The valuation $v_p$ on $L$ extends to
the Gauss valuation on each of these fields. 
We also define
\[	\mathcal{E}_\infty = \mathcal{E} \hat{\otimes}_{W(\mathfrak{K})} 
W(\mathfrak{K}^{alg}), ~~ \text{ and } ~~\mathcal{E}_\infty^\dagger = 
\mathcal{E}_\infty \cap 
\mathcal{E}_{F^{un},L}^\dagger,\]
where $\hat{\otimes}$ denotes the completed tensor product.

\subsection{Logarithmic decay rings}
\label{subsection: logarithmic decay rings}
Let $k \in \frac{1}{e}\Z$. We define the partial valuation $w_k:\mathcal{E}_{F^{un},L} \to \Z\cup \infty$
as follows: for $x = \sum a_nT^n$ we have
\[ w_k(x) = \min_{v_p(a_n)\leq k} \{n\}.\]
Informally, $w_k(x)$ is the smallest power of $T$ occurring in $x$ reduced modulo 
$\pi^{ke+1}$.
These partial valuations satisfy the following inequalities:
\begin{align} \label{Colmez inequalities}
\begin{split}
w_k(x+y) & \geq  \min(w_k(x),w_k(y)), \\
w_k(xy) & \geq  \min_{i+j\leq k} (w_i(x) + w_j(y)).
\end{split}
\end{align}
In either inequality, there is an equality if the minimum is attained exactly once.
For $r > 0$ we define
\begin{align*}
\mathcal{E}^r_\infty &= \Big\{ x  \in 
\mathcal{E}_\infty \Big |  
\begin{array}  {l}
\text{there exists $c>0$ such that} \\
w_k(x) \geq -cp^{kr} \text{ for $k\gg 0$} 
\end{array}
\Big \}, \\
\mathcal{E}^r &= 	\mathcal{E}^r_\infty \cap \mathcal{E}.
\end{align*}
By \eqref{Colmez inequalities} both $\mathcal{E}^r_\infty$ and $\mathcal{E}^r$ are 
rings.
Let $s=\frac{1}{r}$ and let $P\subset \mathbb{R}^2$ denote the lower convex hull of 
the points 
\[(0,0), (s,p), (2s,p^2), \dots .\]
Then $P$ is the graph of a continuous piece-wise linear function 
$f_r:\mathbb{R}_{\geq 
	0} \to \mathbb{R}_{\geq 0}$. For $c>0$ we define
\begin{align*}
\mathcal{O}_{\mathcal{E}_\infty}^{r,c} &= \Big \{ x \in 
\mathcal{O}_{\mathcal{E}_\infty}~\Big |~ 
w_k(x) \geq -cf_r(k),\text{ for }k\geq 0 \Big \} \\
\mathcal{O}_{\mathcal{E}}^{r,c}&= \mathcal{O}_{\mathcal{E}_\infty}^{r,c} \cap 
\mathcal{E}.
\end{align*}

\noindent As $f_r$ is super-additive (i.e. $f_r(x+y)\geq f_r(x)+f_r(y)$ for all $x,y\geq 0$) we know that
$\mathcal{O}_{\mathcal{E}}^{r,c}$ (resp. $\mathcal{O}_{\mathcal{E}_\infty}^{r,c}$)
is a $p$-adically closed subring of $\mathcal{O}_{\mathcal{E}}^r$ (resp. 
$\mathcal{O}_{\mathcal{E}_\infty}^r$). 
%If $L'$ is a ramified extension of $L$ 
%we have
%\begin{align}\label{rc growth with base extension}
%\mathcal{O}_{\mathcal{E}}^{r,c} &= \mathcal{O}_{\mathcal{E}_{F,L'}}^{r,c} \cap 
%\mathcal{E}.
%\end{align}

\begin{lemma} \label{lemma: inverting in r,c ring}
	Let $x \in \mathcal{O}_{\mathcal{E}_\infty}^{r,c}$ with $w_0(x) =0$. Then 
	$x^{-1} \in 
	\mathcal{O}_{\mathcal{E}_\infty}^{r,c}$. 
\end{lemma}
\begin{proof}
	This follows from \eqref{Colmez inequalities} and the super-additivatity of
	$f^r$.
\end{proof}

\begin{proposition} \label{proposition: E^r is a field}
	The rings $\mathcal{E}^r_\infty$ and $\mathcal{E}^r$ are fields.
\end{proposition}
\begin{proof} Let $x\in 
	\mathcal{E}^r$. After multiplying by a power of $p$ and a power of $T$ we may 
	assume that $x \in \mathcal{O}_{\mathcal{E}^r}$ and $w_0(x)= 0$. Then for
	$c$ sufficiently large we have $x \in \mathcal{O}_{\mathcal{E}}^{r,c}$. The 
	proposition follows from Lemma \ref{lemma: inverting in r,c ring}.
\end{proof}

\subsection{Auxillary spaces of power series} \label{subsection: aux rings}
We now introduce subrings of $\mathcal{E}_\infty$ that 
will be used throughout this article. First, extend $\sigma$ (resp. $\nu$) to an 
endomorphism of $\mathcal{E}_\infty$, $\mathcal{E}_\infty^\dagger$, and 
$\mathcal{E}_\infty^r$ (resp. 
$\mathcal{E}_{F,\Q_p}$, $\mathcal{E}_{F,\Q_p}^\dagger$, and $\mathcal{E}_{F,\Q_p}^r$)
sending $T \mapsto T^q$ (resp. $T\mapsto T^p$). We define 
\begin{align*}
\mathcal{E}_{\infty}^{(p^r)} = \Bigg \{ \sum a_nT^n \in 
\mathcal{E}_\infty ~~ \Bigg| ~~ a_n=0 ~~\text{for all }p^r\mid n \Bigg \}.
\end{align*}
Note that $\mathcal{E}_{\infty }^{(q)}$ is a $\sigma(\mathcal{E}_{\infty})$-module. 
Next, for any $m>0$ we define the $p$-adically closed ring
\begin{align*}
\mathcal{O}_{\mathcal{E}_{\infty}}^{\dagger,m} = \Bigg \{ x \in 
\mathcal{O}_{\mathcal{E}_\infty^\dagger} ~~ \Bigg|~~ w_k(x) \geq mk~~\text{for 
	all 
}k\geq 0) \Bigg 
\}.
\end{align*}
Finally, we define the ring
\begin{align*}
\mathcal{E}_{\infty}^{\dagger,-} = \Bigg \{ \sum a_nT^n \in 
\mathcal{E}_\infty^\dagger ~~ \Bigg| ~~ a_n=0 ~~\text{for all }n>0 \Bigg \}.
\end{align*}

\section{$\mathbf{F}$-isocrystals and their slope filtrations} 
\label{section: Local F-isocrystals and their slope filtrations}

\subsection{$(\sigma,\nabla)$-modules}

For this subsection we let $R$ be $\mathcal{E}_*, \mathcal{E}_*^r$, or 
$\mathcal{E}_*^\dagger$,
where $*$ is either $\infty$ or nothing.

\begin{definition}
	A $\sigma$-\emph{module} over $R$ is a finite dimensional vector space $M$ over 
	$R$
	equipped with a $\sigma$-semilinear endomorphism $\varphi: M \to M$
	whose linearization is an isomorphism.  That is,
	we have $\varphi(am)=\sigma(a)\varphi(m)$ for $a\in R$
	and $\sigma^* \varphi: R \otimes_{\sigma} M \to M$
	is an isomorphism.  Given a basis $\mathbf{e}=(e_1, \dots, e_n)$ of $M$,
	there exists a matrix $A \in M_{n\times n}(R)$ such that 
	$\varphi(\mathbf{e})=A\mathbf{e}$.
	This matrix is well-defined up to skew-conjugation. We refer to $A$
	as a \emph{Frobenius matrix} of $M$ or a \emph{Frobenius structure} of $M$.
\end{definition}
\begin{definition}
	Let $\Omega_R$ be the module of differentials of $R$ over $E$. 
	We define the $\delta_T: R \to \Omega_R$ to be the
	map $a(T) \mapsto \frac{da(T)}{dT} dT$.  A $\nabla$-module over $R$
	is a vector space $M$ over $R$ equipped with a connection.  That is,
	$M$ comes with an $E$-linear map $\nabla: M \to M\otimes_R \Omega_R$
	satisfying the Leibnitz rule.
\end{definition}

\begin{definition} By abuse of notation, we let
	$\sigma: \Omega_R \to \Omega_R$ be the map defined by 
	$\sigma(f(T)dT)=\sigma(f(T))d\sigma(T)$. 
	A $(\sigma,\nabla)$-module is a
	$\sigma$-module $M$ with a connection $\nabla$ such that:
	\begin{center}
		\begin{tikzcd}
			
			M \arrow{d}{\varphi} \arrow{r}{\nabla} &
			M \otimes \Omega_R \arrow{d}{\varphi \otimes \sigma} \\
			
			M \arrow{r}{\nabla} & M \otimes \Omega_R ,
		\end{tikzcd}
	\end{center}
	is a commutative diagram. We denote the category of $(\sigma,\nabla)$-modules 
	over $R$ by
	$\Mphiok{R}{\sigma}$
\end{definition}

\begin{remark}
	Let $\sigma_0:R \to R$ be another lift of the $q$-Frobenius morphism 
	that extends $\sigma|_E$. We may define the category $\Mphiok{R}{\sigma_0}$
	in an analogous way. One may show that the 
	categories $\Mphiok{R}{\sigma_0}$ and $\Mphiok{R}{\sigma}$ are equivalent 
	(see \cite[Proposition 3.4.2]{Tsuzuki2} for $R=\mathcal{E},\mathcal{E}^\dagger$ 
	and the case where $R=\mathcal{E}^r$ is similar). However, for the purposes of 
	this article it is enough to only consider $\sigma$. 
\end{remark}

\subsection{\textbf{F}-isocrystals} 
\label{subsection: fisocrystal section}
Let $X$ be $\Spec(\mathfrak{K})$, $\Spec(F)$, or a smooth geometrically connected 
variety over 
$\Spec(\mathfrak{K})$. We 
let 
$\Fisoc^\dagger(X)$ denote the category of
overconvergent $\mathbf{F}$-isocrystals on $X$ 
and we let $\Fisoc(X)$ denote the category of convergent $\mathbf{F}$-isocrystals
on $X$ (see \cite[\S 2]{Kedlaya-notes_on_isocrystals} for 
precise definitions). These categories are $\Q_q$-linear. We define $\Fisoc(X)\otimes 
L$ (resp. 
$\Fisoc^\dagger(X)\otimes L$)
to be the category whose objects are pairs $(M,f)$, where $M$ is an object of 
$\Fisoc(X)$ 
(resp. 
$\Fisoc^\dagger(X)$) and $f:L \to End(M)$ is a $\Q_q$-linear map.
For a finite extension $L'$ of $L$, there is a functor $\Fisoc(X)\otimes L \to 
\Fisoc(X)\otimes L'$ (resp. 
$\Fisoc^\dagger(X)\otimes L \to \Fisoc^\dagger(X)\otimes L'$).
We have the following equivalences of categories:
\begin{align*}
\Mphiok{\mathcal{E}_{F,L}}{\sigma}&\longleftrightarrow \Fisoc(\Spec(F))\otimes L, 
\\
\Mphiok{\mathcal{E}_{F,L}^\dagger}{\sigma}&\longleftrightarrow 
\Fisoc^\dagger(\Spec(F))\otimes L.
\end{align*}
%
%\begin{proposition} \label{equiv of categories for ext of scalars of f-crystals}
%	The category of $(\varphi, \nabla)$-modules over $\mathcal{E}$ 
%	(resp. 
%	$\mathcal{E}^\dagger$ and $\mathcal{E}^r$) is equivalent to
%	$\Fisoc(F)\otimes L$ (resp. 
%	$\Fisoc^\dagger(F)\otimes L$ and $\Fisoc^r(F)\otimes L$).
%\end{proposition}
%
%\begin{proof}
%	Add a sketch later. 
%\end{proof}
%
There is a functor $\iota^\dagger:\Fisoc^\dagger(X)\otimes L \to 
\Fisoc(X)\otimes L$. In terms of $(\sigma,\nabla)$-modules, this functor
sends a $(\sigma,\nabla)$-module $M$ over $\mathcal{E}_{F,L}^\dagger$ to $M 
\otimes_{\mathcal{E}_{F,L}^\dagger} \mathcal{E}_{F,L}$. This functor is known to be 
fully faithful (see \cite{Kedlaya-fully_faithful}). 

\paragraph{Pullbacks}
Let $f:Y\to X$ be a smooth morphism. There are pullback
functors: 
\begin{align*}
f^*: \Fisoc(X)\otimes L &\to  \Fisoc(Y)\otimes L, \\
f^*: \Fisoc(X)^\dagger \otimes L &\to  \Fisoc(Y)^\dagger \otimes L.
\end{align*}
Consider the morphism $\eta: \Spec(F^{un}) \to \Spec(F)$. Then $\eta^*$
sends a $(\sigma,\nabla)$-module $M$ over
$\mathcal{E}_{F,L}$ (resp. $\mathcal{E}_{F,L}^\dagger$) to 
$M\otimes_{\mathcal{E}_{F,L}} \mathcal{E}_{F^{un},L}$  (resp. 
$M\otimes_{\mathcal{E}_{F,L}^\dagger } \mathcal{E}_{F^{un},L}^\dagger$).
In particular, the functor $\eta^*$ factors through the ``tensor by 
$\mathcal{E}_\infty$'' 
(resp. ``tensor by $\mathcal{E}_\infty^\dagger$'') functors.

\paragraph{The log-decay condition}
\label{paragraph: the log decay cond}
Let $K=\mathfrak{K}'((u))$ be a finite separable extension of $F$. An object $M$ of 
$\Fisoc(\Spec(K))\otimes L$ may be realizes as a $(\sigma_u,\nabla)$-module over
$\mathcal{E}_{K,L}$, where $\sigma_u$ sends $u$ to $u^q$. 

\begin{definition}
	Let $r>0$ and let $M$ be an object of $\Fisoc(\Spec(K))\otimes L$. We say that
	$M$ has $r$-log-decay for $u$ if there exists a  $(\sigma_u,\nabla)$-module
	$M^r$ over $\mathcal{E}_{K,L}^r$ such that $M^r \otimes_{\mathcal{E}_{K,L}^r} 
	\mathcal{E}_{K,L} \cong M$. We say that an object $M$ of $\Fisoc(\Spec(F))\otimes 
	L$ has 
	$r$-log-decay if for every finite separable morphism $f:\Spec(\mathfrak{K}'((u))) 
	\to 
	\Spec(F)$, the pullback $f^*M$ has $r$-log-decay for $u$. We say that
	$M$ has strict $r$-log-decay if $M$ has $r$-log-decay and does not have
	$s$-log-decay for any $s<r$.
\end{definition}

\begin{remark}
	Our definition of $r$-log-decay is slightly ad-hoc. One can prove than if 
	an object of $\Fisoc(\Spec(F))\otimes L$ has $r$-log-decay for $T$, then it has
	$r$-log-decay. This intrinsic 
	approach to $r$-log-decay will appear in future work of the author, but is not 
	necessary
	for this article.
\end{remark}

\subsubsection{Unit-root $\mathbf{F}$-isocrystals and $p$-adic representations}
\begin{definition} \label{definition: isoclinic and etale}
	Assume $k$ is algebraically closed. We say an object $M$ of 
	$\Fisoc(\Spec(k))\otimes L$ is \emph{\'etale} or \emph{unit-root} if all of its 
	slopes are zero when viewed as a Dieudonn\`e module (see e.g. \cite[\S 
	1.3]{katz-slope_filtration}). More generally, we say that an object $M$ of 
	$\Fisoc(X)\otimes L$ (resp. $\Fisoc^\dagger(X)\otimes L$) is unit-root 
	if for every geometric point $x$ in $X$, the pullback $x^*M$ is unit-root. We 
	denote the category of unit-root objects by $\Fisoc(X)^{et} \otimes L$ (resp. 
	$\Fisoc(X)^{\dagger, et} \otimes L$). 
\end{definition}

\begin{theorem}[Katz (see \cite{Katz-p-adic_properties} or 
	\cite{Crew-F-isocrystals_padic_reps}), Tsuzuki (see
	\cite{Tsuzuki1})] \label{mod p riemann-hurwitz}
	There is an equivalence of categories
	\begin{align*}
	\Fisoc^{et}(X)\otimes L &\longleftrightarrow \Bigg\{\begin{array}  {l}
	\text{continuous finite dimentional} \\  
	\text{representations} ~ 
	\rho:\pi_1^{et}(X) \to GL_n(L)
	\end{array} \Bigg\} .
	\end{align*}
	If we restrict ourselves to unit-root $\mathbf{F}$-isocrystals over $\Spec(F)$ 
	that
	are overconvergent we obtain:
	\begin{align*}
	\Fisoc^{\dagger, et}(\Spec(F))\otimes L &\longleftrightarrow \{\text{continuous 
		representation} ~ \rho:G_F \to GL_n(L), ~|\rho(I_F)|<\infty\},
	\end{align*}
	where $I_F$ denotes the inertia subgroup of $G_F$.
\end{theorem}
\paragraph{Embeddings of $L$ into $\Q_p^{alg}$}
Let $(M,f)$ be an object of $\Fisoc(X)\otimes L$. For any $g\in \Theta$, we obtain an 
object 
$(M,f\circ g^{-1})$ of $\Fisoc(X) \otimes g(L)$. 
In particular, if $L$ is a Galois extension of $\Q_p$ there is an action of 
$G_{L/\Q_p}$ on $(M,f)$. What does this mean in terms of Galois representations? 
Assume that $(M,f)$ is unit-root
and corresponds to the Galois representation $\rho:\pi_1^{et}(X) \to 
\mathcal{O}_L^\times$. 
Then $(M,f\circ g^{-1})$ corresponds to the composition $g\circ\rho$, which we 
denote by $\rho^g$. 
\paragraph{Pullbacks}
Let $f:\Spec(Y) \to \Spec(X)$ be a finite \'etale morphism. Let 
$M$ be an object of $\Fisoc(X)^{et} \otimes L$ corresponding to the 
representation $\rho$ of $\pi_1^{et}(X)$. Then the pullback $f^* M$ corresponds to 
the pullback
of $\rho$ along the map $\pi_1^{et}(Y) \to \pi_1^{et}(X)$.

\paragraph{Galois representations associated to isoclinic $\mathbf{F}$-isocrystals}
\label{paragraph:galois reps for isoclinic}
\begin{definition}
	Let $\omega \in \Q_p^{alg}$. After enlarging $L$ we may assume that $\omega \in 
	L$. Let $L(\omega)$ denote the object of $\Fisoc(\Spec(\mathfrak{K})) \otimes L$
	whose Frobenius structure is multiplication by $\omega^{-1}$. 
	By abuse of notation, we regard $L(\omega)$ as an object of $\Fisoc(X) \otimes L$.
	We say that an object $M$ of $\Fisoc(X) \otimes L$ is 
	isoclinic of slope $\alpha$ if there exists $\omega \in \Q_p^{alg}$
	with $v_q(\omega)=\alpha$ such that $M\otimes L(\omega)$ is unit-root.
\end{definition}

Now let $M$ be an isoclinic object of $\Fisoc(X)\otimes L$ with slope $\alpha$. Let 
$\omega_1,\omega_2 \in \Q_p^{alg}$ have $q$-adic valuation $\alpha$. Then 
$M_i=M\otimes L(\omega_i)$ are both unit-root $\mathbf{F}$-isocrystals. Let $\rho_i$ 
be the 
representation of $\pi_1^{et}(X)$ corresponding to $M_i$. We have $M_1\cong M_2 
\otimes L(\frac{\omega_1}{\omega_2})$. The $\mathbf{F}$-isocrystal 
$L(\frac{\omega_1}{\omega_2})$ is unit-root, and thus corresponds to a $p$-adic 
character $\chi$ of 
$\pi_1^{et}(X)$.  Note that $L(\frac{\omega_1}{\omega_2})$ descends along the
structure map $X \to \Spec(\mathfrak{K})$. This means that $\chi$ descends to a 
character of $G_\mathfrak{K}$. Thus, we may associate to $M$ a $p$-adic 
representation of $\pi_1^{et}(X)$ that is well-defined up to twist by a character of 
$G_\mathfrak{K}$.

\paragraph{Slope filtrations and log-decay}
\label{subsection: slope filtrations}
Let $M$ be an object of $\Fisoc^\dagger(X)\otimes L$. After replacing $X$ with a 
dense open 
subset, there is a unique slope filtration
\[0=M_0 \subset M_1 \subset M_2 \subset ... \subset 
M_d=\iota^\dagger(M),\]
where each graded piece $gr_i(M)=M_{i}/M_{i-1}$ is isoclinic of
slope $\alpha_i$ and $\alpha_1<\alpha_2<\dots<\alpha_d$ (see 
\cite[Theorem 2.4.2]{katz-slope_filtration} or \cite[\S 
4]{Kedlaya-notes_on_isocrystals}). 
\begin{definition}
	\label{definition: Newton polygons} The Newton polygon $NP(M)$
	of $M$ is the lower convex hull of the points $(rank(M_i), \sum_{i=0}^n rank(gr_i(M))\alpha_i)$ in the $xy$-plane, where
	$i$ ranges from $0$ to $d$. 
\end{definition}
This slope filtration is functorial in $X$.
If $\alpha_1=0$, then $M_1$ is a unit-root convergent $\mathbf{F}$-isocrystal.
In this case, we will denote $M_1$ by $M^{unit}$. The following conjecture relates 
the 
rate
of logarithmic decay of $M_i$ to the differences between consecutive slopes.

\begin{conjecture}  \label{log-decay and slope filtrations}
	Assume that $M$ is irreducible. Take $X=\Spec(F)$ and let $r_i = 
	\frac{1}{\alpha_{i+1} - \alpha_i}$.
	Then $M_i$ has strict 
	$r_i$-log-decay.
\end{conjecture}
\noindent In \S \ref{Section: main results}
we provide evidence for this conjecture.
For example, we prove that $M_i$ has $r_i$-log-decay
for any $M$. We also prove Conjecture \ref{log-decay and slope filtrations}
when the slopes of $M$ are integers and $L=\Q_p$.

\section{Ramification theory for $p$-adic characters}
\label{section: ramification}
\subsection{Ramification and $p$-adic Lie filtrations}
For any $k \geq 0 $ we define
\begin{align*}
O_k^\times &= \{x \in \mathcal{O}_L^\times ~~|~~v_p(1-x)> k\}, \\
O_k^+ &= \{ x \in \mathcal{O}_L^+ ~~|~~v_p(x) > k\}.
\end{align*}	
Consider a multiplicative character $\rho: G_F \to \mathcal{O}_L^\times$
(resp. an additive character $\psi: G_F \to \mathcal{O}_L^+$). We let
$s_k(\rho)$ (resp. $s_k(\psi)$) denote the largest upper numbering ramification break
of the extension corresponding to $\rho^{-1}(O_k^\times ) \subset G_F$ (resp. 
$\psi^{-1}(O_k^+) \subset G_F$). The following lemma is immediate:

\begin{lemma}   \label{ramification lemma: add to mult}
	Let $\psi:G_F \to \mathcal{O}_L^+$ be an additive character
	with $\psi(G_F) \subset O_{1}^+$. Then 
	$s_k(\psi) = s_k(\exp(\psi))$ for all $k \geq 1$.
	Similarly, let $\rho: G_F \to \mathcal{O}_L^\times$ be a multiplicative
	character with $\rho(G_F)\subset O_1^\times$. Then $s_k(\rho) = s_k(\log(\rho))$
	for all $k \geq 1$.
\end{lemma}
%\begin{lemma} \label{ramification lemma: tensor powers}
%	Let $\rho$ be a multiplicative character with $\rho(G_F)\subset O_k^\times$
%	for some $k\geq 0$.
%	Then $\rho^{\otimes p} (G_F) \subset O_{k+1}^\times$ and for all $n\in \Q_{\geq 
%	0}$ 
%	we have
%	\begin{align*}
%		s_k(\rho) &= s_{n+1}(\rho^{\otimes p}).
%	\end{align*}
%\end{lemma}

\noindent The image $\psi(G_F) \subset \mathcal{O}_L^+$ is a free $\Z_p$-module whose
rank $d$ is at most $[L:\Q_p]$. We may break up $\psi$ as $\psi=\sum \psi_i$,
such that the image of $\psi_i$ is a rank one $\Z_p$-module. Note that
\begin{align} \label{breakup ramification into zp towers eq}
s_k(\psi) &= \max \{s_k(\psi_i)\}.
\end{align}
This follows from the ``quotient property'' of 
the upper ramificaiton numbers (see, e.g., \cite[Proposition 
14]{Serre-local_fields}). Finally, we mention
a natural restriction on the growth of $s_k(\psi)$ and $s_k(\rho)$.

\begin{lemma} \label{monodromy growth restriction lemma}
	Let $\rho$ (resp. $\psi$) be a multiplicative (resp. additive) character. 
	For all $k \geq 0$ we have $s_{k+1}(\rho)\geq p s_k(\rho)$
	and $s_{k+1}(\psi) \geq p s_k(\psi)$. 
\end{lemma}

\begin{proof}
	Let $F_\infty/F$ be the Abelian $p$-adic Lie extension
	corresponding to $\psi$. By local class field theory, $F_\infty/F$ corresponds
	to an open subgroup $H$ of $\mathcal{O}_F^\times$ such that 
	$\mathcal{O}_F^\times / H \cong G_{F_\infty/F}$.
	The image of the subgroup $U^s = 1 + T^s\mathcal{O}_F$ in $\mathcal{O}_F^\times / 
	H$ corresponds to the 
	subgroup $G_{F_\infty/F}^s$.  For any group
	$A$ we let $A^{\times p}$ denote the $\{ a^p~~|~a \in A\}$.  Since we are in 
	characteristic $p$ we have
	$(U^s)^{\times p} \subset U^{ps}$ for any $s$.  On the Galois side of the 
	correspondence this means $(G_{F_\infty/F}^s)^{\times p} \subset 
	G_{F_\infty/F}^{ps}$.
	However, since $G_{F_\infty/F}$ is isomorphic to $\Z_p^d$ for some $d$, we know 
	that 
	$(G_{F_\infty/F}^{s_{k}(\rho)})^{\times p} = G_{F_\infty/F}^{s_{k+1}(\rho)}$.
	The multiplicative case is identical.
\end{proof}

\begin{corollary} \label{corollary: no small r bounded monodromy}
	Let $r<1$. If $\rho$ has $r$-bounded monodromy, then $\rho$ has finite 
	monodromy.
\end{corollary}

\begin{definition}
	Let $\rho$ be a multiplicative character. We say $\rho$ is \emph{harshly 
		ramified} if for $k \gg 0$ we have $s_{k+1}(\rho) > ps_k(\rho)$.
\end{definition}

\subsection{Ramification and base change}

\begin{proposition} \label{base change ramification proposition}
	Let $\rho: G_F \to \mathcal{O}_L^\times$ be a multiplicative character.
	Let $K$ be a finite Galois extension of $F$. There exists $c\geq 0$ such that 
	for $k$ 
	sufficiently large 
	we have
	\begin{align*}
	s_k(\rho|_{G_K}) &= |G_{K/F}^0|s_k(\rho) - c.
	\end{align*}
	Furthermore, $c=0$ if and only if $K$ is tamely ramified over $F$.
\end{proposition}

\begin{proof}
	Recall the inverse Hasse-Herbrand function
	\begin{align}
	\psi_{K/F}(s) &= \int_0^s [G_{K/F}^0:G_{K/F}^x]dx. \label{hasse-herbrand eq}
	\end{align}
	Let $F_\infty$ (resp. $K_\infty$) be the fixed field of $\rho^{-1}(\{1\})$
	(resp. $\rho|_{G_K}^{-1}(\{1\})$). We claim
	that for $s$ sufficiently large, the map $G_{K_\infty/K} \to G_{F_\infty/F}$
	restricts to an isomorphism
	\begin{align} \label{upper numbering base change}
	G_{K_{\infty}/K}^{\psi_{K/F}(s)} &\cong G_{F_\infty/F}^s. 
	\end{align}
	To see this, let $K_1$ be a finite Galois extension of $F$
	containing $K$.
	Using \cite[\S 4, Proposition 15]{Serre-local_fields},
	we see that
	\begin{align*}
	G_{K_1/F}^s \cap G_{K_1/K} &= (G_{K_1/F})_{\psi_{K_1/F}(s)} \cap G_{K_1/K} \\
	&= (G_{K_1/K})_{\psi_{K_1/F}(s)} \\
	&= (G_{K_1/K})_{\psi_{K_1/K} \circ \psi_{K/F}(s)} \\
	&= (G_{K_1/K})^{\psi_{K/F}(s)}.
	\end{align*}
	By taking a limit along the finite subextensions of $K_\infty$ over $F$ we obtain
	\begin{align} \label{upper numbering intersection relation}
	G_{K_\infty/F}^s \cap G_{K_\infty/K} &= G_{K_\infty/K}^{\psi_{K/F}(s)}.
	\end{align}
	Since $K/F$ is a finite extension we have 
	$G_{K_\infty/F}^s \subset G_{K_\infty/K}$ for large $s$. Also, for $s$ large the 
	restriction 
	map
	$G_{K_\infty/F}^s \to G_{F_\infty/F}^s$ is an isomorphism. Combining this with
	\eqref{upper numbering intersection relation} proves \eqref{upper numbering base 
		change}.
	From \eqref{upper numbering base change} we see that 
	$s_k(\rho|_{G_K})=\psi_{K/F}(s_k(\rho))$ for $k$ sufficiently large. The
	proposition follows from \eqref{hasse-herbrand eq}.
\end{proof}

\begin{corollary} \label{corollary: harsh ramification base change}
	Assume $\rho(I_F)$ is infinite. If $K$ is wildly ramified over $F$, then 
	$\rho|_{G_K}$ is harshly 
	ramified. 
\end{corollary}
\begin{proof}
	This follows from Lemma \ref{monodromy growth restriction lemma}
	and Proposition \ref{base change ramification proposition}.
\end{proof}

\begin{corollary} \label{corollary: pseudostable monodromy after base change}
	Adopt the notation from Proposition \ref{base change ramification 
		proposition}. Then $\rho$ has $a$-stable monodromy (resp. $r$-bounded 
	monodromy) if and only if $\rho|_{G_K}$ has $a$-stable monodromy (resp. 
	$r$-bounded monodromy).
\end{corollary}

\section{Monodromy of rank one $\mathbf{F}$-isocrystals}
\label{section: Monodromy section}

\subsection{Frobenius structures of $p$-adic characters} 
\label{subsection: p-adic characters}
\subsubsection{Rings of periods}
Let $\widetilde{\mathcal{E}}=L\otimes_{W(\mathbb{F}_{q})} W(F^{alg})$.
There is an embedding $\iota: 
\mathcal{E} \hookrightarrow 
\widetilde{\mathcal{E}}$  that sends $T$ to the Teichmuller lift $[T]$.
Recall that $\sigma: \mathcal{E} \to \mathcal{E}$ (resp. $\nu: \mathcal{E} \to \mathcal{E}$) is the endomorphism that sends $T$ to $T^q$ (resp. $T$ to  $T^p$) and restricts to endomorphism $\sigma$ (resp. $\nu$) of $L$ as defined in \S \ref{section: conventions}.
If $K$ is a finite separable extension of $F$, there exists a unique unramified 
extension $\mathcal{E}^K$ of $\mathcal{E}$
contained in $\widetilde{\mathcal{E}}$ whose residue field is $K$ (see 
\cite[Theorem 2.2]{Matsuda-local_index}).
Define
\begin{align*}
\mathcal{E}^{un} &= \bigcup_{[K:F]<\infty}\mathcal{E}^{K}, 
\end{align*}
and let $\widetilde{\mathcal{E}}^{un}$ be the $p$-adic
completion of $\mathcal{E}^{un}$. Note that $G_F$ acts continuously on 
$\widetilde{\mathcal{E}}^{un}$. 
\subsubsection{Multiplicative characters}
\label{subsubsection: mult case} 
Let $\rho: G_F \to \mathcal{O}_L^\times$ be a multiplicative character
and let $V=e_0 L$ be a one dimensional vector space over $L$ on which $G_F$ acts 
through $\rho$.  The corresponding object $M$ of $\Fisoc^{et}(\Spec(F))\otimes L$ is 
$
(\widetilde{\mathcal{E}}^{un} \otimes_{L} V)^{G_F}$,
where the Frobenius acts by 
$\sigma \otimes \text{id}$.
Thus $M$ consists of elements $x_0\otimes e_0 $ such that $\frac{x_0}{x_0^\gamma} = 
\rho(\gamma)$.
The Frobenius structure
of $M$ is given by
$m=\frac{x_0^\sigma}{x_0}$.
Note that $m$ is well-defined up to multiplication
by elements of the form $\frac{c^\sigma}{c}$, where $c \in \mathcal{E}^\times$. 

Let $V_0=\Q_p e_0$, so that $V=V_0 \otimes_{\Q_p} L$. Assume that $\rho$ 
factors
through a map $\rho_0: G_F \to GL(V_0)$ and let $M_0$ be the corresponding object of 
$\Fisoc^{et}(\Spec(F))$. There is an isomorphism
\begin{align*}
(\widetilde{\mathcal{E}}^{un} \otimes_{\Q_p} V_0)^{G_F} \to 
(\widetilde{\mathcal{E}}^{un} \otimes_{L} V)^{G_F}
\end{align*}
sending $x_0 \otimes_{\Q_p} e_0$ to $x_0 \otimes_L e_0$. Then
$m_0=\frac{x_0^{\nu}}{x_0}$ is a Frobenius structure
of $M_0$ 
and we have $m = \prod\limits_{i=0}^{f-1} m_0^{\nu^i}$
\subsubsection{Additive characters} \label{subsubsection: additive characters}
Let $\psi:G_F \to p\mathcal{O}_L^+$ be an additive character. 
For some $y_0 \in\widetilde{\mathcal{E}}^{un}$ and $a \in \mathcal{E}$ we have
\begin{align*}
y_0 - y_0^{\gamma} &= \psi(\gamma), \\
y_0^\sigma - y_0 &= a.
\end{align*}
We refer to $a$ as the \emph{Frobenius structure}
of $\psi$. It is well-defined up to addition by elements $c^\sigma - c$ for $c \in 
\mathcal{E}$. If we set $\rho =\exp (\psi)$, then we may take $y_0=\log(x_0)$ and 
$a=\log(m)$
(here $x_0$ and $m$ are as in \S \ref{subsubsection: mult case}).
If $\psi$ factors through $\psi_0: G_F \to p \Z_p^+$,
we obtain
\begin{align*}
y_0^\nu - y_0 &= \log(m_0), \\
\log(m)&= \sum_{i=0}^{f-1} \log(m_0)^{\nu^i}.
\end{align*}
%
%Finally, let $B \subset \mathcal{O}_L$ freely generate $\psi(G_F)\subset 
%p\mathcal{O}_L^+$ as a $\Z_p$-module. We may decompose $\psi$ as 
% $\psi=\sum_{b \in B} b\psi_{b}$, where $\psi_b

\subsubsection{Maximal Frobenius structures}
Let $x \in \mathcal{O}_{\mathcal{E}_{F^{un},L}}$. For $k \geq 0$, define
the $k$-th weighted partial valuation as follows:
\begin{align*}
c_k(x) &= \min_{\stackrel{i \in \Z_{\geq 0}}{i\leq k}}\{ p^i w_{k-i}(x)\} \cup \{0\}.
\end{align*}
These weighted partial valuations satisfy the following properties:
\begin{lemma} \label{lemma: weighted partial valuations}
	Let $x,y \in \mathcal{O}_{\mathcal{E}_{F^{un},L}}$. The following hold:
	\begin{enumerate}[label=(\roman*)]
		\item For $k\geq 0$ we have $c_{k+1}(px) = c_k(x)$, $c_k(x^\nu)=pc_k(x)$, and 
		$c_k(x^\sigma)=qc_k(x)$. \label{weighted partial prop 1}
		\item We have $c_k(x+y) \geq \min(c_k(x),c_k(y))$. If the minimum is attained 
		exactly once there is equality. \label{weighted partial prop 2}
		\item We have $c_k(x\cdot y) \geq c_k(x) + c_k(y)$.  \label{weighted partial 
			prop 3}
		\item If $v_p(x-1)\geq 1$ (resp. $v_p(x)\geq 1$), then $c_k(x)=c_k(\log(x))$ 
		(resp. $c_k(x)=c_k(\exp(x))$) for all $k\geq 1$.  \label{weighted partial 
			prop 4}
	\end{enumerate}
\end{lemma}
\begin{proof}
	Statements \ref{weighted partial prop 1}-\ref{weighted partial prop 3}
	follow from the definition and the statement about 
	exponentials will 
	follow from the statement about logs. Write $x=1+py$, with $y \in 
	\mathcal{O}_{\mathcal{E}_{F^{un},L}}$. It is enough to prove 
	$c_k(\frac{p^ny^n}{n})>c_k(py)$ for all $n\geq 2$. Let
	$m=v_p(n)$. Then by \ref{weighted partial prop 1} and \ref{weighted partial prop 
		3}, we see 
	\begin{align*}
	c_k\Big (\frac{p^ny^n}{n} \Big) &\geq\frac{n}{p^{n-1-m}} c_k(py) \\
	&> c_k(py).
	\end{align*}
\end{proof}

\begin{definition}
	Let $x \in \mathcal{O}_{\mathcal{E}_{F^{un},L}}$. We say that $x$ is 
	\emph{maximal}
	if the following holds: for all $k\geq 0$, we have $q|c_{k+1}(x)$ if and only if 
	$c_{k+1}(x)=pc_k(x)$.
\end{definition}
\noindent The term maximal is justified by the following Proposition:
\begin{proposition} \label{Proposition: minimal definition works}
	The following holds:
	\begin{enumerate}[label=(\roman*)]
		\item Let $x \in \mathcal{O}_{\mathcal{E}_{F^{un},L}}$ be maximal. If
		$y \in \mathcal{O}_{\mathcal{E}_{F^{un},L}}$, then
		$c_k(x)\geq c_k(x+y^\sigma - y)$ for all $k\geq 0$.  \label{maximal prop 1}
		\item Let $x \in 1 + p\mathcal{O}_{\mathcal{E}_{F^{un},L}}$ be maximal. If $y 
		\in 1 + 
		p\mathcal{O}_{\mathcal{E}_{F^{un},L}}$, then $c_k(x)\geq 
		c_k(x\frac{y^\sigma}{y})$ for 
		all $k\geq 0$. \label{maximal prop 2}
	\end{enumerate}
\end{proposition}
\begin{proof}
	By Lemma \ref{lemma: weighted partial valuations} it is enough to prove 
	\ref{maximal prop 
		1}. Let $z=x+y^\sigma - y$ and let $k\geq 0$ be the smallest value 
	with $c_k(z)>c_k(x)$. Since $x$ is maximal, we have $c_k(x) < 0$. By Lemma 
	\ref{lemma: weighted partial valuations} we see that
	$c_k(y^\sigma)=c_k(x)$. Thus $q|c_k(x)$, so $c_k(x)=pc_{k-1}(x)\geq pc_{k-1}(z)$.
	This implies $c_k(z) > pc_{k-1}(z)$,
	which is impossible. 
\end{proof}

\begin{definition}
	Let $\rho:G_F \to \mathcal{O}_L^\times$ be a multiplicative character.
	A \emph{maximal Frobenius structure} of $\rho$ is a Frobenius structure
	$m$ that is maximal. A $\Theta$-\emph{maximal Frobenius structure} of $\rho$
	is a set $\{m_g\}_{g \in \Theta}$, where $m_g$ is a maximal Frobenius structure
	of $\rho^g$. We make analogous definitions for additive characters. 
\end{definition}

\begin{corollary} \label{corollary: all maximal frobenius are equal}
	Let $\rho: G_F \to 1+p\mathcal{O}_L$ (resp. $\psi: G_F \to \mathcal{O}_L^+$)
	be a multiplicative (resp. additive) character. Let $\alpha_1, \alpha_2$ be a 
	maximal
	Frobenius structure of $\rho$ (resp. $\psi$). Then $c_k(\alpha_1)=c_k(\alpha_2)$
	for all $k\geq 0$.
\end{corollary}
\noindent It will be helpful to distinguish when a Frobenius structure is maximal. 
This motivates the following:
\begin{lemma} \label{lemma: non divisble are maximal}
	If $x \in \mathcal{O}_{\mathcal{E}_\infty^{(q)}}$, then $x$ is maximal. 
\end{lemma}

\begin{remark}
	Let $\rho: G_F \to 1+p\mathcal{O}_L$ be a multiplicative character. One may show that $\rho$ has
	a Frobenius structure contained in $\mathcal{O}_{\mathcal{E}_\infty^{(q)}}$ (start with a Frobenius
	structure $\alpha \equiv 1 \mod p$ and successively find a Frobenius structure that looks like
	it is in $\mathcal{O}_{\mathcal{E}_\infty^{(q)}}$ modulo powers of $\pi$).
	Thus, by Lemma \ref{lemma: non divisble are maximal} we know that $\rho$ has a maximal Frobenius structure. 
\end{remark}

%	The following proposition summarizes everything from \S \ref{subsection: p-adic 
%	characters}
%	\begin{proposition}
%		Let $\rho: G_F \to 1+p\mathcal{O}_L$ be a multiplicative character
%		and let $M_\rho$ be the corresponding object of $\Fisoc(F) \otimes L$. There 
%		exists a Frobenius structure $\alpha$ that is $f$-minimal. For any other 
%		$f$-minimal Frobenius structure $\alpha'$ of $M_\rho$, we
%		have $c_k(\alpha)=c_k(\alpha')$ for all $k\in \Q_{\geq 0}$.
%		Furthermore, if $\rho$ factors through a character
%		$\rho_0:G_F \to 1+p\Z_p$ and $\alpha_0$
%	\end{proposition}

\subsection{The additive situation}
For this subsection we assume $\mathfrak{K}$ is a finite field.
\begin{proposition} \label{main additive monodromy theorem}
	Let $\psi: G_F \to \mathcal{O}_L^+$ be an additive character
	and let $\{a_g\}_{g \in \Theta}$ be a $\Theta$-maximal Frobenius structure. Then 
	for
	$k\geq 0 $ we have
	\begin{align} \label{additive monodromy eq}
	\frac{q}{p}s_k(\psi) &= \max_{g \in \Theta} \{-c_k(a_g)\}.
	\end{align}		
\end{proposition}

\noindent We will deduce this proposition from a theorem due to Kosters and Wan:

\begin{theorem} \label{Kosters-Wan theorem} (Kosters-Wan, see 
	\cite[Proposition 3.3 or \S 4.1]{Kosters-Wan})
	Let $\psi: G_F \to p\Z_p^+$ be an additive
	character that surjects onto $p\Z_p^+$. Then there exists a maximal Frobenius $a\in \mathcal{O}_{\mathcal{E}_{F,\Q_p}}$ of $\psi$ 
	and
	$-c_k(a)=s_k(\psi)$ for all $k\geq 0$.
\end{theorem}
\begin{proof}
	Let $F_\infty$ be the fixed field of $\ker(\psi)$. Then $F_\infty/F$ is a $\Z_p$-extension, and thus
	by Artin-Schreier-Witt theory corresponds to an equivalence class of $W(F)/(1-\nu)W(F)$ (here $\nu$
	denotes the standard $p$-power Frobenius endomorphism on the ring of $p$-typical Witt vectors). 
	By \cite[Proposition 3.1]{Kosters-Wan}, for any $b \in \mathfrak{K}$ with $Tr_{\mathfrak{K}/\mathbb{F}_p}(b)\neq 0$,
	there exists a unique representative $a_0$ in $z$ of the form
	\begin{align*}
		a_0 &= r[b]+ \sum_{i\geq 1,(i,p)=1} r_i T^{-i},
	\end{align*} 
	where $r \in \Z_p$ and $r_i \in W(\mathfrak{K})$ (here we regard $\mathcal{E}$ as a subring of $W(F)$ 
	via the map $\iota$ defined at the beginning of \S \ref{subsection: p-adic characters}). Note that
	by Lemma \ref{lemma: non divisble are maximal}, the element $a_0$ is maximal.
	As in \S \ref{subsubsection: additive characters}, from $a_0$ we obtain a surjective character $\psi_0: Gal(F_\infty/F) \to \Z_p$ (take $y_0$ satisfying
	$y_0^\nu - y_0 = a_0$ and then set $\psi_0(g)=y_0^g - y_0$). 
	Thus, for some $b \in \Z_p$ we have $\psi=b\psi_0$, so that $a=ba_0$ is a maximal Frobenius structure of $\psi$. 
	A proposition due to Kosters-Wan \cite[Proposition 3.3]{Kosters-Wan} states that $s_n(\psi_0)$ is equal to 
	$-c_n(a_0)$ (note that \cite{Kosters-Wan} state their result in terms of conductors,
	but this easily translates into a statement about higher ramification groups). From here we deduce that
	$s_n(\psi)=-c_n(a)$. 
\end{proof}

\begin{corollary} \label{Kosters-Wan base change}
	Let $\psi: G_F \to \mathcal{O}_L^+$ be an additive
	character such that $\psi(G_F)\subset b \Z_p^+$ for some $b \in 
	\mathcal{O}_L^+$.
	Let $a$ be a maximal Frobenius structure of $\psi$. Then $\frac{q}{p}s_k(\psi) = 
	-c_k(a)$
	for all $k\geq 0$.
\end{corollary}
\begin{proof}
	Let $\psi_0:G_F \to \Z_p^+$ be the character defined
	by $b^{-1}\psi$ and let $a_0$ be a maximal Frobenius of $\psi_0$. 
	By Lemma \ref{lemma: weighted partial valuations} we see that $b\sum\limits_{i=0}^{f-1} a_0^{\nu^i}$ is a Frobenius structure of 
	$\psi$. The result follows
	from Corollary \ref{corollary: all maximal frobenius are equal} and Theorem 
	\ref{Kosters-Wan theorem}.
\end{proof}

\begin{definition} \label{definition: nice}
	Let $M\subset \mathcal{O}_L$ be a 
	$\Z_p$-module and let $B$
	be a basis of $M$. For $i=0,\dots, e-1$,
	define $B_i=\{b \in B | v_p(b) \equiv \frac{i}{e} \mod \Z\}$. In particular,
	we may write $B_i=\{\pi^i p^{r_{i,1}} u_{i,1}, \dots, \pi^i p^{r_{i,d_i}}u_{i,d_i}\}$. We say $B$ is
	\emph{nice} if for each $i$, the reductions of 
	$u_{i,1}, \dots, u_{i,d_i}$ modulo $\pi$ are linearly
	independent over $\mathbb{F}_p$. 
\end{definition}
\begin{lemma} \label{lemma: nice basis exists}
	Every $\Z_p$-module $M \subset \mathcal{O}_L$
	has a \emph{nice} basis. 
\end{lemma}
\begin{lemma} \label{linear combinations of nice basis}
	Let $B$ be a nice basis of $M$. For every
	$b\in B$, let $x_b$ be an element of $W(\mathfrak{K})$. Then
	\begin{align*}
	\min_{b \in B} \{ v_p(x_b b) \}&= 
	\min_{g \in \Theta}\Big \{ v_p\Big (\sum_{b \in B} x_b b^{g}\Big )\Big \}.
	\end{align*}
\end{lemma}

\begin{proof}
	We may replace $L$ with a field that is Galois over $\Q_p$. Let 
	$r=\min_{b 
		\in B}(v_p(x_bb))$ and write $r=\frac{i}{e}+n$ with
	$0\leq i < e$. Let $B'\subset B$ consist of all $b\in B$ such that
	$v_p(x_bb)=r$. Note that $B'$ is a subset of $B_i$ from Definition \ref{definition: nice}. For each $b\in B'$ we write $b=u_b\pi^i p^{r_b}$
	and $x_b=v_b p^{s_b}$, where $u_b\in \mathcal{O}_L^\times$ and
	$v_b \in W(\mathfrak{K})^\times$. Let $\overline{A}$
	be a matrix whose columns are $[\overline{u}_b,\overline{u}_b^p, \dots, 
	\overline{u}_b^{p^{f-1}}]^T$ for each $b \in B'$.
	By the definition of nice, we know that the rank of 
	$\overline{A}$ is $|B'|$. Thus, there exists $j$ such that
	\begin{align*}
	\sum_{b \in B'} \overline{u_b}^{p^j} \overline{v_b} & \neq 0.
	\end{align*}
	Let $g$ be an element of $G_{L/\Q_p}$ that reduces to the $p^j$-th
	power map on $\mathcal{O}_L/\pi\mathcal{O}_L$ and let $c 
	=\frac{g(\pi^i)}{\pi^i}$. Then
	\begin{align*}
	\sum_{b \in B} x_b b^g &\equiv c\pi^ip^n \sum_{b \in B'} 
	\overline{u_b}^{p^j} \overline{v_b} \not \equiv 0\mod \pi^{er+1}.
	\end{align*}
\end{proof}

\begin{corollary} \label{linear combination of power series using nice basis}
	Consider a subset $\{y_b\}_{b \in B} \subset \mathcal{O}_{\mathcal{E}_{F,\Q_p}}$. 
	There exists $g \in \Theta$ such that
	\begin{align*}
	\max_{b \in B} \{-c_k( y_b b) \} &= -c_k\Big (\sum_{b \in B}  y_b b^g \Big ). 
	\end{align*} 
\end{corollary}

\begin{proof}
	Let $n=\min_{b \in B} w_k( y_b b)$ and let
	$x_b\in W(\mathfrak{K})$ denote the coefficient of $T^n$ in $y_b$. Note that
	$\min_{b \in B} v_p(x_b b) \leq k$. 
	By Lemma
	\ref{linear combinations of nice basis}, there exists $g \in \Theta$ with 
	$v_p(\sum_{b \in B} x_b b^{g})\leq k$.
	It follows that the $p$-adic valuation of the coefficient
	of $T^n$ in $\sum_{b \in B} y_b b^{g}$ is less than $k$. 
\end{proof}

\begin{proof} (Of Proposition \ref{main additive monodromy theorem})
	Let $B$ be a nice basis 
	of $\psi(G_F)$. Decompose
	$\psi$ as $\sum_{b \in B} b\psi_b$, where 
	the image of $\psi_b$ is $\Z_p$. For each $b \in B$, let $y_b \in 
	\mathcal{O}_{\mathcal{E}_{F,\Q_p}}$
	be a maximal Frobenius of $\psi_b$. Then $a_g = \sum  y_b b^g$
	is a maximal Frobenius of $\psi^g$ for each $g \in \Theta$.
	By \eqref{breakup ramification into zp towers eq},
	Corollary \ref{Kosters-Wan base change}, and  Corollary \ref{linear combination 
	of power series using nice basis}
	we have
	\begin{align*}
	\frac{q}{p}s_k(\psi) &= \max_{b \in B}\Big\{ \frac{q}{p}s_k(b\psi_b) \Big 
	\}\\
	&=\max_{b \in B}  \{-c_k(y_b b) \}\\
	&= \max_{g \in \Theta} \{-c_k(a_g) \}.
	\end{align*}
\end{proof}

\subsection{The multiplicative situation}
Again, we assume $\mathfrak{K}$ is a finite field for this subsection.
\begin{proposition} \label{main monodromy theorem}
	Let $\rho: G_F \to 1+p\mathcal{O}_L$ and let $\{m_g\}_{g 
		\in \Theta}$ be a $\Theta$-maximal Frobenius structure of $\rho|_{F^{ur}}$. 
		Then for
	$k\geq 0$, we have
	\begin{align*} 
	\frac{q}{p}s_k(\rho) &= \max_{g \in \Theta} \{-c_k(m_g)\}.
	\end{align*}	
\end{proposition}

\begin{proof}
	By Lemma \ref{lemma: weighted partial valuations} we know that
	$\{\log(m_g)\}_{g \in \Theta}$ is a $\Theta$-maximal Frobenius structure
	for $\log \circ \rho$ and that $c_k(m_g) = c_k(\log(m_g))$ for all $k\geq 0$. The 
	proposition follows from Lemma \ref{ramification lemma: add to mult} and 
	Proposition \ref{main additive monodromy theorem}.
\end{proof}	
\begin{corollary} \label{corollary: harshly ramified monodromy}
	Adopt the notation of Proposition \ref{main monodromy theorem} and
	assume $\rho$ is harshly ramified. For
	$k$ large, we have
	\begin{align*} 
	\frac{q}{p}s_k(\rho) &= \max_{g \in \Theta} \{ -w_k(m_g)\}.
	\end{align*}	
\end{corollary}

\begin{corollary} \label{corollary: small log decay}
	Let $\rho$ be a multiplicative character of $G_F$ and let $M$ be the 
	corresponding 
	unit-root $\mathbf{F}$-isocrystal. Let $r>0$ and assume that
	$M^g$ has $r$-log-decay 
	for each $g\in \Theta $. If $r\geq 1$, then $\rho$ has
	$r$-bounded monodromy. If $r<1$, then $M$ is overconvergent
	and $\rho$ has finite monodromy.
\end{corollary}
\begin{proof}
	After replacing $F$ with a finite extension we may assume 
	$\rho(G_F) \subset 1+p\mathcal{O}_L$. We may also assume that $\rho$ is 
	either harshly ramified or unramified by Corollary \ref{corollary: harsh 
	ramification base change}. Let $\{m_g\}_{g
		\in \Theta}$ be a $\Theta$-maximal Frobenius structure. There exists $c>0$ 
		such that $\{m_g\}_{g
		\in \Theta} \subset \matcaloe{r,c}$. By 
	Proposition \ref{main monodromy 
		theorem}, we know $\rho$ has $r$-bounded monodromy. 
	The statement about $r<1$ follows from Corollary \ref{corollary: no small r bounded monodromy}.
\end{proof}

\section{Recursive Frobenius equations}
\label{section: recursive frobenius} 
For this section, we define 
$E_a$ to be $\Frac{\mathcal{O}_L \otimes_{\mathbb{F}_q} W(\mathbb{F}_{q^a})}$ and 
$\mathcal{E}_a$ (resp $\mathcal{E}_a^\dagger$) to be 
$\mathcal{E}_{\mathbb{F}_{q^a}((T)),L}$ (resp. 
$\mathcal{E}_{\mathbb{F}_{q^a}((T)),L}^\dagger$). 
Note that $\mathcal{E}_{\infty}$ is the closure of $\bigcup_{a=1}^\infty 
\mathcal{E}_a$
and $\bigcup_{a=1}^\infty \mathcal{E}_a^\dagger$ is dense in 
$\mathcal{E}_\infty^\dagger$. 

\subsection{Basic definitions}
Let $A \in M_{d\times d}(\mathcal{O}_{\mathcal{E}_\infty})$ with $v_q(A)> 0$
and let $C \in M_{d \times 1}(\mathcal{O}_{\mathcal{E}_\infty})$. We 
define $R_a(A;C)$ to be the unique $d \times 1$ matrix satisfying the 
recursive 
Frobenius 
equation: 
\begin{align} \label{recursive def}
x&=Ax^{\sigma^a} +C.
\end{align}  
We define $R(A;C)$ to be $R_1(A;C)$. These solutions have the following explicit 
formula:
\begin{align} \label{frobenius solution unraveled}
R_a(A;C) &=\sum_{i=0}^\infty A^{\frac{1-\sigma^{ai}}{1-\sigma^a}}C^{\sigma^{ai}}.
\end{align}
More generally, for $A_1,\dots, A_m,B_1,\dots, B_{m-1} \in M_{d\times 
	d}(\mathcal{O}_{\mathcal{E}_{\infty}})$ with $v_q(A_i)>0$, we give the recursive 
definition:\[R_a(A_1,\dots,A_m;B_1,\dots,B_{m-1},C) 
=R_a(A_1,\dots,A_{m-1};B_1,\dots,B_{m-1}\cdot 
R_a(A_m;C)^\sigma).\]
If the $A_i$ are all equal to $A$, we define:
\begin{align*}
R_a(A; B_1,\dots,B_{m-1},C) &=R_a(A_1,\dots,A_m;B_1,\dots,B_{m-1},C).
\end{align*} 
When $d=1$ we will drop the pretext of dealing with matrices and view everything 
as elements of $\mathcal{O}_{\mathcal{E}_{\infty}}$. 
\begin{lemma} \label{addition formula}
	We have
	\begin{align}
	R_a(A,B+C) &= R_a(A,B)+R_a(A,C) \label{Frob eq additive formula 1}\\
	R_a(A+B,C) &= \sum_{m=0}^\infty R_a(A; \underbrace{B,\dots,B}_{m \text{ times}}, 
	C).  
	\label{Frob eq additive formula 2}
	\end{align}
\end{lemma}
\begin{proof}
	The first equation is immediate. Define $
	S_m=R_a(A; \underbrace{B,\dots,B}_{m \text{ times}}, 
	C)$ and note that
	\begin{align*}
	S_m&=\begin{cases}
	AS_m^{\sigma^a} + BS_{m-1}^{\sigma^a} & m\geq 1 \\
	AS_0^{\sigma^a} + C & m=0.
	\end{cases}
	\end{align*}
	Thus, $\sum S_m$ is the solution to the Frobenius equation $x=(A+B)x^{\sigma^a} + 
	C$, which proves \eqref{Frob eq additive formula 2}.
\end{proof}

\subsection{Growth of recursive Frobenius equations}
We begin by proving some basic properties of 
$\mathcal{O}_{\mathcal{E}_\infty}^{r,c}$, 
which will be used throughout the rest of the article. 

\begin{lemma} \label{dividing by p^s for rc growth}
	Let $r,c>0$ and let $\omega \in \mathcal{O}_L$ satisfy $v_q(\omega)=\frac{1}{r}$.
	\begin{enumerate}[label=(\roman*)]
		\item If $x \in  \mathcal{O}_{\mathcal{E}_\infty}^{r,c}$ then $\omega x \in  
		\mathcal{O}_{\mathcal{E}_\infty}^{r,q^{-1}c}$ and $x^\sigma \in 
		\mathcal{O}_{\mathcal{E}_\infty}^{r,qc}$. \label{mult rc eq1}
		\item If $y=\omega x \in 
		\mathcal{O}_{\mathcal{E}_\infty}^{r,c}$ and $w_0(x)\geq 0$, then $x \in 
		\mathcal{O}_{\mathcal{E}_\infty}^{r,qc}$ and $y^{1+\sigma+\dots+\sigma^n} \in 
		\mathcal{O}_{\mathcal{E}_\infty}^{r,c}$. \label{mult rc eq2}
		\item Let $x$ and $y$ be as in \ref{mult rc eq2} and let $z \in 
		\mathcal{O}_{\mathcal{E}_\infty}^{r,c}$. Then $yz^\sigma \in 
		\mathcal{O}_{\mathcal{E}_\infty}^{r,c}$. \label{mult rc eq3}
	\end{enumerate}
	
\end{lemma}
\begin{proof}
	Parts \ref{mult rc eq1}-\ref{mult rc eq2} follow from the definition of 
	$\mathcal{O}_{\mathcal{E}_\infty}^{r,c}$. To prove \ref{mult rc eq3}, note that
	$xz^\sigma \in \mathcal{O}_{\mathcal{E}_\infty}^{r,qc}$. 
\end{proof}

\begin{corollary} \label{rc growth recursive formula}
	Let $y=\omega x \in 
	\mathcal{O}_{\mathcal{E}_\infty}^{r,c}$ and assume that $w_0(x)\geq 0$.
	If $z \in \mathcal{O}_{\mathcal{E}_\infty}^{r,c}$, then $R(y,z)\in 
	\mathcal{O}_{\mathcal{E}_\infty}^{r,c}$.
\end{corollary}

\subsubsection{Approximating Frobenius equations} 
\label{subsubsection approximating frobenius equations}
For this section we fix $c>0$ and $c>c_1>0$. 
\begin{lemma} \label{lemma: existence of growth divisibility property}
	There exists $N_0>0$, depending only on $c$ and $c_1$,
	such that:
	\begin{enumerate}[label=(\roman*)]
		\item If $x,y \in \mathcal{O}_{\mathcal{E}_\infty}^{1,c}$ and
		$q^{N_0}| x,y$, then $xy^\sigma \in 
		\mathcal{O}_{\mathcal{E}_\infty}^{1,c_1}$.\label{lemma: 
			existence 
			of growth divisibility property eq1}
		\item Let $x \in \mathcal{O}_{\mathcal{E}_\infty}^{1,c}\cap q\mathcal{O}_{\mathcal{E}_\infty}$ and let $y \in 
		\mathcal{O}_{\mathcal{E}_\infty}^{1,c_1} \cap 
		q^{N_0}\mathcal{O}_{\mathcal{E}_\infty}$. Then $xy^\sigma \in 
		\mathcal{O}_{\mathcal{E}_\infty}^{1,c_1}\cap 
		q^{N_0}\mathcal{O}_{\mathcal{E}_\infty}$. 
		\label{lemma: 
			existence of growth divisibility property eq2}
	\end{enumerate}
\end{lemma}
\begin{proof} Let $x,y \in \mathcal{O}_{\mathcal{E}_\infty}^{1,c}$ and assume
	$q^{N_0}| x,y$. Write $x=q^{N_0-1}u$ and $y=q^{N_0-1}v$. By Lemma 
	\ref{dividing by p^s for rc growth}, we know
	$u \in \mathcal{O}_{\mathcal{E}_\infty}^{1,q^{N_0-1}c}$ and $v^\sigma \in 
	\mathcal{O}_{\mathcal{E}_\infty}^{1,q^{N_0}c}$. Thus, $uv^\sigma \in 
	\mathcal{O}_{\mathcal{E}_\infty}^{1,q^{N_0}c}$. 
	From Lemma \ref{dividing by p^s for rc growth} we see that
	$xy^\sigma \in \mathcal{O}_{\mathcal{E}_\infty}^{1,q^{-N_0+2}c}$. In particular,
	if $N_0$ is large enough \ref{lemma: existence of growth divisibility property 
		eq1} holds. The proof of \ref{lemma: 
		existence of growth divisibility property eq2} is similar.
\end{proof}

\noindent For the remainder of this subsection, we fix $N_0>0$ as in Lemma 
\ref{lemma: existence of growth divisibility property}.

\begin{lemma} \label{lemma: break up overconvergent things}
	Let $x \in \mathcal{O}_{\mathcal{E}^\dagger_\infty }\cap 
	\mathcal{O}_{\mathcal{E}_\infty}^{1,c}$. Then  for some $a\geq 0$ we may write
	$x=y+z$, where $y \in \mathcal{O}_{E_a}((T))\cap 
	\mathcal{O}_{\mathcal{E}_\infty}^{1,c}$ and $z \in 
	\mathcal{O}_{\mathcal{E}_\infty}^{1,c_1}\cap 
	q^{N_0}\mathcal{O}_{\mathcal{E}_\infty}$.
\end{lemma}

\begin{proof}
	Write $x=y+z$, where $z \in 
	\mathcal{O}_{\mathcal{E}_\infty}^{1,c_1}\cap 
	q^{N_0}\mathcal{O}_{\mathcal{E}_\infty}$ and $y$ is a Laurent series with a 
	finite pole. Since $\cup \mathcal{O}_{\mathcal{E}_n^\dagger}$ is dense in 
	$\mathcal{O}_{\mathcal{E}^\dagger_\infty }$, we may take $y$ to lie in $ 
	\mathcal{O}_{E_a}((T))$ for $a$ sufficiently large. Furthermore, we know $y\in 
	\mathcal{O}_{\mathcal{E}_\infty}^{1,c}$, as both $x$ and $z$ lie in 
	$\mathcal{O}_{\mathcal{E}_\infty}^{1,c}$.
\end{proof}

\begin{lemma} \label{lemma: effect of recursive on highly divisible r,d things}
	Let $A$ be a matrix such that $qA\in M_{d\times 
		d}(\mathcal{O}_{\mathcal{E}_\infty}^{r,c})$ and $w_0(A)\geq 0$. Let
	$C\in M_{d\times 1}(\mathcal{O}_{\mathcal{E}_\infty}^{1,c})$ and $Y_1$ (resp. 
	$Y_2$) be an $d\times d$ (resp. $d \times 1$) matrix with entries 
	in
	$\mathcal{O}_{\mathcal{E}_\infty}^{1,c_1}\cap 
	q^{N_0}\mathcal{O}_{\mathcal{E}_\infty}$. Then,
	\begin{align}
	R(qA+Y_1,C)&\equiv R(qA,C) \mod \mathcal{O}_{\mathcal{E}_\infty}^{1,c_1}\cap 
	q^{N_0}\mathcal{O}_{\mathcal{E}_\infty} \label{effect of recursive eq1}\\
	R(qA,Y_2+C) &\equiv R(qA,C) \mod \mathcal{O}_{\mathcal{E}_\infty}^{1,c_1}\cap 
	q^{N_0}\mathcal{O}_{\mathcal{E}_\infty}\label{effect of recursive eq2}.		
	\end{align}
	
\end{lemma}
\begin{proof}
	Let $Z_n=(qA)^{\frac{1-\sigma^n}{1-\sigma}} 
	Y_2^{\sigma^n} $. We have $Z_n\in 
	\mathcal{O}_{\mathcal{E}_\infty}^{1,c_1}\cap 
	q^{N_0}\mathcal{O}_{\mathcal{E}_\infty}$ by Lemma \ref{lemma: existence of growth 
		divisibility property} and the relation $Z_{n+1}=AZ_n^\sigma$.
	Then \eqref{effect of recursive eq2} follows from \eqref{frobenius solution 
		unraveled} and \eqref{Frob eq additive formula 1}. To prove \eqref{effect of 
		recursive 
		eq1}, by Lemma \ref{addition formula} 
	it is enough to show
	$R(qA; \underbrace{Y_1,\dots,Y_1}_{m \text{ times}}, 
	C)$ has entries in $\mathcal{O}_{\mathcal{E}_\infty}^{1,c_1}\cap 
	q^{N_0}\mathcal{O}_{\mathcal{E}_\infty}$ for $m\geq 1$.  This follows from 
	Lemma \ref{lemma: existence of growth divisibility property}.
\end{proof}

\subsection{Spaces of recursive Frobenius solutions}
\label{subsection: spaces of recursive frobenius solutions}
A \emph{tuple} $\lambda$ will be taken to mean a finite tuple of negative integers 
$\lambda=(m_1,\dots,m_r)$. We define $\mathbf{len}(\lambda)$ to be
$r$, i.e. the length of the tuple. 
For any $c \in \Z_{\geq 1}$ we let $c\lambda$ denote the tuple 
obtained by scalar multiplication. We say $\lambda $ is $q^a$-\emph{prime} if 
$q^a\nmid m_i$ for each $i\geq 1$. If $\mathbf{len}(\lambda)\geq 1$ we define
\begin{align*}
S_a(\lambda) &=R_a(q^a; T^{m_1},\dots, T^{m_r}).
\end{align*} 
and if $\mathbf{len}(\lambda)=0$ we set $S_a(\lambda)=1$. For $\mu_0,\mu_1 \in 
\Z_{<0}$ 
we have the following relations:
\begin{align}
S_a(\mu_0, \lambda) &= T^{\mu_0} \cdot S_a(q^{a}\lambda) + q^a S_a(q^a\mu_0,q^a 
\lambda) 
\label{S eq 1}, \\
S_a(\mu_0,\mu_1,\lambda) &= S_a(\mu_0+q^{a}\mu_1, q^a\lambda) + q^a 
S_a(\mu_0,q^a\mu_1,q^{a}\lambda), 
\label{S eq 2}
\end{align}
which follow from \eqref{recursive def}. Finally, we define the following 
$E_{a}[T^{-1}]$-modules:
\begin{align*}
\mathcal{N}_a &= \Big\{\sum_{i=1}^n b_i S_a(\lambda_i)~|~b_i \in 
E_a[T^{-1}]
\text{ and the }\lambda_i\text{ are tuples}\Big\}, \\
\mathcal{N}_a^{(p)} &= \Big\{\sum_{i=1}^n b_i S_a(\lambda_i)~|~b_i \in 
E_a[T^{-1}]
\text{ and the }\lambda_i\text{ are $q^a$-prime tuples}\Big\}. 
\end{align*}

\begin{lemma} \label{prime space is the same as nonprime space}
	We have $\mathcal{N}_a^{(p)} = 
	\mathcal{N}_a$.
\end{lemma}
\begin{proof}
	We proceed 
	by induction on $r=\mathbf{len}(\lambda)$. When $r=1$ the result follows from 
	\eqref{S eq 1}. Let $r\geq 1$ and assume the result holds for 
	all $k<r$. Write $\lambda 
	=(\mu_0,\mu_1,\lambda_0)$, where $\mathbf{len}(\lambda_0)=r-2$. By our inductive 
	hypothesis we may assume that 
	$(\mu_1,\lambda_0)$ is $q^a$-prime. If $q^a|\mu_0$, from \eqref{S eq 1} and 
	\eqref{S 
		eq 2} we obtain
	\begin{align*}
	S_a(\mu_0,\mu_1,\lambda_0) &= S_a(\mu_0+q^a\mu_1,q^a\lambda_0) - 
	T^{\mu_0}S_a(\mu_1,q^a\lambda_0) 
	+ S_a(\frac{\mu_0}{q^a},\mu_1,\lambda_0). 
	\end{align*}
	Both $S_a(\mu_0+q^a\mu_1,q^a\lambda_0)$ and $ T^{\mu_0}S_a(\mu_1,q^a\lambda_0)$ 
	are contained in 
	$\mathcal{N}_a^{(p)}$ by our inductive hypothesis, so it suffices to prove 
	$S_a(\frac{\mu_0}{q^a},\mu_1,\lambda_0)\in \mathcal{N}_a^{(p)}$. Repeating this 
	argument proves the lemma. 
\end{proof}

\begin{lemma} \label{lemma: recursive relation for powers of frobenius}
	Let $x \in \mathcal{E}_\infty$ and let $a,b$ be integers with $b=a\cdot c$. We 
	have
	\begin{align*}
	R_a(q;x) &= \sum\limits_{i=0}^{c-1} q^{ia} R_b(q^b;x^{\sigma^{ia}})
	\end{align*}
\end{lemma}

\begin{lemma} \label{lemma: recursive lands in Na}
	Let $C \in M_{d\times 1}(\mathcal{E}_\infty)$ have entries in $\mathcal{N}_a$ and
	let $B\in T^{-1}M_{d\times d}(E_a[T^{-1}])$.
	Then $R(q1_d;BC)$ has entries in $\mathcal{N}_a$.
\end{lemma}
\begin{corollary} \label{corollary: recursive solutions with finite pole are in Ma}
	Let $B_1,\dots, B_m \in T^{-1} M_{d\times d}(E_a[T^{-1}])$ and 
	$C \in T^{-1}M_{d\times 1}(E_a[T^{-1}])$. Then we have
	$R(q1_d;B_1,\dots,B_m,C)$ has entries in $\mathcal{N}_a$.
\end{corollary}

\begin{proposition}\label{main $r$-recursive Prop}
	Let $A \in q1_d+q^2T^{-1}M_{d\times 
		d}(\mathcal{O}_{\mathcal{E}_{\infty}^{\dagger,-}})$ and $C \in 
	qT^{-1}M_{d \times 1}(\mathcal{O}_{\mathcal{E}_{\infty}^{\dagger,-}} )$. 
	Let $c_1>0$ and let $N_0$ be a positive integer.
	Then for $a$ sufficiently large, there exists $Z_{c_1}\in M_{d\times 
		1}(\mathcal{E}_\infty)$ with entries in $\mathcal{N}_a^{(p)}$ such that
	\begin{align*}
	R(qA,C) &\equiv Z_{c_1} \mod \mathcal{O}_{\mathcal{E}_\infty}^{1,c_1}\cap q^{N_0} 
	\mathcal{O}_{\mathcal{E}_\infty}.
	\end{align*}

\end{proposition}

\begin{proof}
	Write $A$ as $q1_d + q^2 A_1$. There exists $c>c_1$ 
	such that
	$qA_1$ and $C$ have entries in $\mathcal{O}_{\mathcal{E}_\infty}^{1,c}$. After 
	increasing $N_0$, we 
	may assume that the results 
	of \S \ref{subsubsection approximating frobenius equations} hold.
	By Lemma \ref{lemma: break up overconvergent things} 
	and Lemma \ref{lemma: effect of recursive on highly divisible r,d things}, 
	we may also assume the entries of $qA_1$ and $C$ lie in 
	$T^{-1}\mathcal{O}_{E_a}[T^{-1}]\cap \mathcal{O}_{\mathcal{E}_\infty}^{1,c}$.
	Then by Corollary \ref{rc growth recursive formula}, we know that
	\begin{align*}
	R(q1_d; \underbrace{qA_1,\dots,qA_1}_{m \text{ times}},C)
	\end{align*}
	is contained in $M_{d\times 1}(\mathcal{O}_{\mathcal{E}_\infty}^{1,c})$ for all 
	$m$.
	Thus, for $m_0$ sufficiently large, we know that
	\begin{align*}
	R(q1_d; \underbrace{q^2A_1,\dots,q^2A_1}_{m \text{ times}},C) &= q^mR(q1_d; 
	\underbrace{qA_1,\dots,qA_1}_{m \text{ times}},C)
	\end{align*}
	has entries in $\mathcal{O}_{\mathcal{E}_\infty}^{1,c_1}\cap q^{N_0} 
	\mathcal{O}_{\mathcal{E}_\infty}$ for all $m>m_0$. From Lemma \ref{addition 
		formula} we obtain
	\begin{align*}
	R(A,C) &\equiv  \sum_{m=0}^{m_0} R(q1_d; \underbrace{q^2A_1,\dots,q^2A_1}_{m 
		\text{ times}},C_0) \mod \mathcal{O}_{\mathcal{E}_\infty}^{1,c_1}\cap q^{N_0} 
	\mathcal{O}_{\mathcal{E}_\infty}.
	\end{align*}
	The right side of this equivalence has entries in $\mathcal{N}_a^{(p)}$
	by Lemma \ref{prime space is the same as nonprime space}
	and Corollary \ref{corollary: recursive solutions with finite pole are in Ma}.
\end{proof}

\subsection{Stable growth for solutions of Frobenius equations}
\label{subsection: stable growth for solutions of Frobenius equations}
In this subsection, we study the growth of certain solutions to recursive Frobenius 
equations.

\begin{definition} \label{a-growth definition}
	Let $S$ be a finite subset of $\mathcal{O}_{\mathcal{E}_\infty}$. Let $a$ be a 
	positive integer and let $s=v_p(q^a)$. We say that $S$ has \emph{$a$-stable 
		growth} if 
	for any $k \in [0,s]$, there exists $m_k$ and $b_k$ such that
	\begin{align*}
	\min_{x \in S}( w_{k+sn}(x))&=m_kq^{an} + b_k 
	\end{align*}
	for $n\gg 0$. We say that $S$ has \emph{stable growth} if $S$ has $a$-stable 
	growth for some $a$. 
\end{definition}

\begin{lemma} \label{union of stable growth}
	For $i=1,2$, let $S_i \subset\mathcal{O}_{\mathcal{E}_\infty}$ have $a_i$ stable 
	growth. 
	Then $S_1 \cup S_2$ has $\lcm(a_1,a_2)$-stable growth. 
\end{lemma}
\begin{proof}
	Observe that a set with $a$-stable growth has $ad$-stable 
	growth for $d\in \Z_{\geq 1}$.
\end{proof}
\noindent To state our main result of this subsection, we define the 
following spaces:
\begin{align*}
\begin{split} 
\mathcal{M}_a &= \mathcal{N}_a \otimes_{E_a[T^{-1}]} 
E_a((T)) \\
\mathcal{M} &= \bigcup_{a=1}^\infty \mathcal{M}_a.
\end{split}
\end{align*}

\begin{proposition} \label{proposition: main recursive growth result}
	Let $S \subset  \mathcal{M}$ and assume there exists $c_1>1$ such that
	\[\min_{x \in S} \{w_k(x)\} \leq -c_1 p^k\]
	for all $k\geq 0$. Then $S$ has stable growth.
\end{proposition}

\noindent The proof of Proposition \ref{proposition: main recursive growth result} 
will be 
broken into several lemmas.

\begin{lemma} \label{lemma: recursive growth condition}
	Let $a\in \Z_{\geq 1}$ and $s=v_p(q^a)$.
	Fix $x \in 
	\mathcal{O}_{\mathcal{E}_\infty}$ and $k\geq 0$. Assume that
	$q^a\nmid w_k(x)$ and that
	\begin{align}\label{recursive growth condition eq1}
	w_{k+sn}(x) &> q^{an}w_k(x)
	\end{align}
	for $n\in \Z_{\geq 1}$. Then for $n\in \Z_{\geq 0}$ we have
	\begin{align*}
	w_{k+sn}(R_a(q^a,x))&=q^{an}w_k(R_a(q^a,x)).
	\end{align*}
\end{lemma}
\begin{proof}
	We proceed by induction on $n$. The case where $n=0$. Set $n\geq 
	1$. Let $z=R_a(q^a,x)$, so that $z=q^az^{\sigma^a} + 
	x$. Note that $w_k(z)\leq w_k(x)$, since $q^a \nmid w_{k}(x)$ and 
	$q^a|w_k(q^az^{\sigma^a})$.
	We have
	\begin{align*}
	w_{k+sn}(z) &\geq \min(w_{k+sn}(q^az^{\sigma^a}), w_{k+sn}(x)) \\
	&= \min (q^aw_{k+s(n-1)}(z),w_{k+sn}(x)) \\
	&= \min(q^{an}w_k(z),w_{k+sn}(x)).
	\end{align*}
	By 
	\eqref{recursive growth condition eq1} we know that $q^{an} w_k(z)< 
	w_{k+sn}(x)$, which proves the result.
\end{proof}

\begin{lemma} \label{lemma: q-prime r-recursive}
	Let $a$ and $s$ be as in Lemma \ref{lemma: recursive growth condition}. Let 
	$\beta_1,\dots, \beta_r \in 
	\mathcal{O}_{E_a}$ and let $\Lambda=\{\lambda_1, \dots, \lambda_r\}$ be a set of 
	$q^a$-prime tuples. We set $x=\sum \beta_i S_a(\lambda_i)$.
	Then for any $k \in [0,s]$, there exists $m_k$ such that 
	\begin{align} \label{q-prime r-recursive eq}
	w_{k+sn}(x) &= m_kq^{an}
	\end{align}
	for $n$ sufficiently large.
\end{lemma}
\begin{proof}
	Define $\mathbf{L}(\Lambda)$ to be $\max\{{\mathbf{len}(\lambda_i)}\}$. We 
	proceed by induction on $\mathbf{L}(\Lambda)$. When
	$\mathbf{L}(\Lambda)$ is $1$, then   
	$x=R_a(q^a,y)$
	with $y \in \mathcal{O}_{E_a}[T^{-1}] \cap \mathcal{E}_\infty^{(q^a)}$ and the 
	lemma 
	follows from Lemma \ref{lemma: recursive growth condition}. 
	Now assume the proposition holds for all collections of $q^a$-prime tuples 
	$\Lambda'$ such that $\mathbf{L}(\Lambda')\leq m$. Let 
	$\Lambda=\{\lambda_1,\dots, \lambda_r\}$ be a collection of $q^a$-prime tuples 
	with $\mathbf{L}(\Lambda)=m+1$. There exist tuples $\Lambda'=\{\lambda_1',\dots, 
	\lambda_r'\}$ and negative integers $\{\kappa_1, \dots, \kappa_r\}$ 
	such that $\mathbf{L}(\Lambda')=m$ and $\lambda_i=(\kappa_i, 
	\lambda_i')$. For each $d<0$ we define 
	\begin{align*}
	y_d &= \begin{cases}
	0 & \text{ if }d \neq \kappa_i \text{ for every }i \\
	\sum\limits_{\kappa_i=d} \beta_iS_a(\lambda_i') & \text{ otherwise}
	\end{cases}.
	\end{align*}
	Let $d_{1},\dots, d_{t}$ be the values of $d$ for which $y_{d}\neq 0$.
	Then we have
	\begin{align*}
	x&= \sum_{i=1}^r R_a(q^a;\beta_iT^{\kappa_i}S_a(\lambda_i')) \\
	&= \sum_{j=1}^t R_a(q^a;T^{{d_j}}y_{d_j}).
	\end{align*}
	By our inductive assumption, there exists $m_{j,k}$ such that 
	$w_{k+sn}(y_{d_j})=q^{an}m_{j,k}$ 
	for $n\gg 0$. Without loss of generality we may assume $d_1 + q^{an} m_{1,k} < 
	\min\limits_{j\geq 2} \{d_j + q^{an} m_{j,k} \}$
	for $n \gg 0$. This holds because the $d_j$ are distinct. Thus,
	$w_{k+sn}(x) = d_1 + q^{an} m_{1,k}$,
	for large $n$. The proposition follows from Lemma \ref{lemma: 
		recursive growth condition}.
\end{proof}

\begin{proof}[Proof of Proposition \ref{proposition: main recursive growth result}]
	By Lemma \ref{union of stable growth}, it is enough to prove the proposition
	for $S=\{x\} \subset \mathcal{M}_a$. After multiplying $x$ by a 
	large power of $qT$ we may write 
	$x=\sum_{i=1}^r b_i S_a(\lambda_i)$, where $b_i \in \mathcal{O}_{E_a}[[T]]$ and 
	$\lambda_i$ are $q^a$-prime. Let
	$\mathcal{A} \subset \mathcal{O}_{\mathcal{E}_a}$ be the 
	$\mathcal{O}_{E_a}$-module 
	generated by $S_a(\lambda_1),\dots,S_a(\lambda_r)$. For $N$ 
	sufficiently large, we know $p^N\mathcal{A} \subset 
	\mathcal{O}_{\mathcal{E}_a}^{1,c_1}$. Let $X=\{x_1,\dots,x_t\} \subset 
	\mathcal{A}$ 
	be a system 
	of representatives of $\mathcal{A} / p^N \mathcal{A}$. Then
	\begin{align} 
	x &= \sum_{i=0}^\infty y_i T^i \notag \\
	&\equiv \sum_{i=0}^\infty z_i T^i \mod 
	\mathcal{O}_{\mathcal{E}_a}^{1,c_1},\label{representatives with good growth}
	\end{align}
	where $y_i \in \mathcal{A}$ and $z_i \in X$ with $y_i \equiv 
	z_i \mod p^N\mathcal{A}$. After reorganizing \eqref{representatives with good 
		growth}, we obtain
	\begin{align*}
	x &\equiv \sum_{j=1}^t x_jf_j(T) \mod \mathcal{O}_{\mathcal{E}}^{1,c_1},
	\end{align*}
	where $f_j(T)\in T^{d_j} + 
	T^{d_j+1}\mathcal{O}_{E_a}[[T]]$ and the $d_j$ are distinct. By Lemma \ref{lemma: 
		q-prime r-recursive}, 
	there 
	exists $m_{j,k}$ such that for $n$ large we have
	\begin{align*} 
	w_{k+sn}(x_jf_j(T)) = d_j+ m_{j,k} q^{an}.
	\end{align*}
	For $n$ sufficiently large, the values $d_j+ m_{j,k} q^{an} $ are distinct.
	Without loss of generality we may
	assume that $d_{1}+ m_{1,k}q^{an}< \min\limits_{j\geq 2} \{d_j+ m_{j,k} 
	q^{an}\}$, 
	which 
	implies 
	$w_{k+sn}(x) = d_{1}+ m_{1,k}q^{an}$ for $n$ large.
\end{proof}

\section{Growth properties of the slope filtration} 
\label{section: solving the unit-root subspace}
\subsection{Local setup} \label{subsection: local setup}
Let $M$ be a rank $n$ object of $\Fisoc^\dagger(F)\otimes L$ with
slope filtration:
\begin{align} \label{slope filtration written out}
0=M_0 \subset M_1 \subset M_2 \subset ... \subset 
M_d=\iota^\dagger(M),
\end{align}
where each graded piece $gr_i(M)=M_{i}/M_{i-1}$ is isoclinic of
slope $\alpha_i$ and has rank $n_i$ (see \S \ref{subsection: slope filtrations}). 
After replacing $L$ with
a finite ramified extension, we may assume there exists $\omega_i \in L$
with $v_q(\omega_i)=\alpha_i$. 
From \eqref{slope filtration written out}, we know that there is a Frobenius 
matrix $A_0$ of $\iota^\dagger(M)$ of the form
\begin{align} \label{matrix eq with filtration}
A_0 &= \begin{pmatrix}
\omega_1 A_1 & * &  *&\dots & * \\
0 & \omega_2 A_{2} &  *& \dots & * \\
0&0& \omega_3 A_3 & \dots & *\\
\vdots & \vdots & \vdots&\ddots & \vdots \\
0 & 0  & 0& \dots & \omega_{d} A_d
\end{pmatrix}
\end{align}
where $\omega_i A_i$ is the Frobenius structure of $gr_i(M)$. Since $gr_i(M) \otimes 
L(\omega_i)$ 
is unit-root, we may assume that $A_i$ has entries in $\mathcal{O}_{\mathcal{E}}$. 
This follows from the construction of the Frobenius structure from the corresponding 
Galois representation (see, e.g., \cite[\S 4]{Katz-p-adic_properties}).

\subsection{The shape of the Frobenius structure of $M$}
%We now study the Frobenius structure of $M$. We first show that the Frobenius 
%matrix can be approximated by a diagonal matrix after enlarging $F$. We then show 
%that the Frobenius structure of $M_0$ has logarithmic decay. Finally, we prove 
%specific properties about the Frobenius matrix when the slopes are all integers. 

\subsubsection{Logarithmic growth of the slope filtration}
We now show that $M_1$ has a Frobenius matrix with log-decay entries. The rate of 
log-decay depends on the difference of the first two slopes.
\begin{lemma} \label{lemma: log growth and slope filtrations}
	Let $A$ be a Frobenius matrix of $M$. Let $N > \alpha_d$ and let 
	$r=\frac{1}{\alpha_2-\alpha_1}$. There exists 
	$C \in M_{n\times n}(\mathcal{O}_{\mathcal{E}^r})$ such that $CAC^{-\sigma}$ is 
	of 
	the form $\begin{pmatrix}
	\omega_1 A_{1,1} & q^N A_{1,2} \\ 
	0 & \omega_2 A_{2,2}
	\end{pmatrix}$,
	where $A_{1,1}, A_{1,2}$ and,$A_{2,2}$ have entries in 
	$\mathcal{O}_{\mathcal{E}^r}$ and $A_{1,1}$ is an $n_1\times n_1$ matrix.
\end{lemma}

\begin{proof}
	After replacing $M$ with $M \otimes L(\omega_1)$, we may assume 
	$\alpha_1=0$ and $\omega_1=1$. Consider the Frobenius matrix $A_0$ of 
	$\iota^\dagger(M)$ from \eqref{matrix eq with filtration}. We may conjugate $A_0$ 
	by a matrix with powers of $q$ along the diagonal, so that each $*$ in the 
	upper-right is divisible by $q^N$. Next, we skew-conjugate \eqref{matrix eq with 
		filtration} by a matrix with powers of $T$ along the diagonal so that 
	$w_0(A_1^{-1})\geq 0$. Let $A$ be a Frobenius matrix of $M$. There exists $B 
	\in 
	GL_n(\mathcal{O}_{\mathcal{E}})$ such that $BAB^{-\sigma} = A_0$. By 
	approximating $B$ with some $B_1$ in $GL_n(\mathcal{O}_{\mathcal{E}^\dagger})$
	and skew-conjugating $A$,
	we may assume that
	\begin{align} \label{simple frobenius congruence equation}
	A &= \begin{pmatrix} A_{1,1} & q^NA_{1,2} \\ q^N A_{2,1} & 
	\omega_2 A_{2,2} \end{pmatrix},
	\end{align}
	where $A_{1,1} \equiv A_1 \mod q^N$. For $c$ sufficiently large,
	the entries of $A_{1,1}^{-1}, q^NA_{1,2}, q^N A_{2,1} $, and 
	$\omega_2 A_{2,2}$ are contained in
	$\mathcal{O}_{\mathcal{E}}^{r,c}$. 
	We will
	show inductively that there exists $C_k=\begin{pmatrix} 1_{n_1} & 0 \\ 
	C_{2,1,k} & 1_{n-n_1}  \end{pmatrix}$ such that:
	\begin{enumerate} [label=(\roman*)]
		\item $A_k=C_k A C_k^{-\sigma}$ is of the form
		$\begin{pmatrix} A_{1,1,k} & q^N A_{1,2} \\ \omega_2^k A_{2,1,k} & 
		\omega_2 A_{2,2,k} \end{pmatrix}$. \label{matrix condition 1 for unit-root}
		\item  The entries of $\omega_2^k A_{2,1,k}$, $\omega_2 A_{2,2,k}$, and 
		$A_{1,1,k}^{-1}$ are contained in 
		$\mathcal{O}_{\mathcal{E}}^{r,c}$. \label{matrix condition 2 for unit-root}
		\item For all $k$ we have $C_k \equiv C_{k-1} \mod \omega_2^k$. 
		\label{matrix condition 4 for unit-root}
	\end{enumerate}
	The result will follow by taking $C=\lim\limits_{k \to \infty} C_k$.
	When $k=1$ this follows from \eqref{simple frobenius congruence equation}.
	Now let $k\geq 1$ and assume $C_{k}$ exists. We define
	$D_k = \begin{pmatrix} 1 & 0 \\ -A_{1,1,k}^{-1}\omega_2^k A_{2,1,k} & 1 
	\end{pmatrix}$
	and set $C_{k+1}=D_kC_k$. It is immediate that
	\ref{matrix condition 1 for unit-root} and \ref{matrix 
		condition 4 for unit-root} are satisfied. We verify \ref{matrix 
		condition 2 for unit-root}
	using Lemma \ref{dividing by p^s for rc growth}.
\end{proof}

\subsubsection{Approximating the Frobenius structure by diagonal matrices}

\begin{lemma} \label{lemma: diagonalization mod p to the N}
	Let $N>\alpha_d$. After replacing $F$ with a finite extension, we may assume that 
	$M$ has a Frobenius matrix of the form
	\begin{align} \label{Frobenius structure congruence}
	A &\equiv \begin{pmatrix}
	\omega_1 1_{n_1} & 0 &0 & \dots & 0 \\
	0 & \omega_2 1_{n_2} & 0 & \dots & 0 \\
	\vdots & \vdots & \ddots & \vdots & \vdots \\
	0 & 0 & 0 & \dots & \omega_{d} 1_{r_d}. 
	\end{pmatrix}  \mod q^N
	\end{align}
\end{lemma}

\begin{proof}
	Let $\rho_i:G_F \to GL_{n_i}(\mathcal{O}_L)$ denote the 
	representation corresponding
	to $gr_i(M) \otimes L(\omega_i)$. After replacing $F$ with a finite separable 
	extension, we may assume that $\rho_i(G_F) \subset 1_{n_i} + q^N M_{n_i 
		\times n_i}(\mathcal{O}_L)$ for each $i$. In particular, there exists a 
	Frobenius matrix $A_i$ of $gr_i(M) \otimes L(\omega_i)$ such that
	\begin{align} \label{isoclinic pieces}
	A_i \equiv 1_{n_i} \mod q^N. 
	\end{align}
	As in the proof of Lemma \ref{lemma: log growth and slope filtrations}, we may 
	assume the 
	$*$'s in \eqref{matrix eq with filtration} are divisible by 
	$q^N$. Then
	\eqref{isoclinic pieces} gives 
	\begin{align*}
	A_0 &\equiv \begin{pmatrix}
	\omega_1 1_{n_1} & 0 &0 & \dots & 0 \\
	0 & \omega_2 1_{n_2} & 0 & \dots & 0 \\
	\vdots & \vdots & \ddots & \vdots & \vdots \\
	0 & 0 & 0 & \dots & \omega_{d} 1_{r_d}
	\end{pmatrix}  \mod q^N.
	\end{align*}
	Let $A$ be a Frobenius matrix of $M$. There exists a matrix $B \in 
	GL_n(\mathcal{E})$ such that $B A B^{-\sigma} = A_0$. The lemma follows by 
	taking $B_0\in GL_n(\mathcal{E}^\dagger)$ whose entries are sufficiently 
	close to $B$. 
\end{proof}

\subsubsection{The case of integer slopes} \label{subsubsection: integer slopes}
We now assume that $\alpha_i=i-1$ and fix $N>\alpha_d$. Assume that $M$ has a 
Frobenius matrix as in Lemma 
\ref{lemma: diagonalization mod p to 
	the N}. We may write:
\begin{align} \label{frobenius matrix 2}
A &= \begin{pmatrix}
A_{1,1 } & q^N A_{1,2}  & q^N A_{1,3} & \dots & q^N A_{1,d} \\
q^N A_{2,1} & qA_{2,2} & q^N A_{2,3} & \dots & q^N A_{2,d} \\
q^N A_{3,1} & q^N A_{3,2} & q^2A_{3,3}  & \dots & q^N A_{3,d} \\
\vdots & \vdots & \vdots & \ddots & \vdots \\
q^N A_{d,1} & q^NA_{d,2} & q^N A_{d,3} & \dots & q^{d-1} A_{d,d}
\end{pmatrix},
\end{align}
where $A_{i,j} \in M_{n_i\times n_j}(\mathcal{O}_{\mathcal{E}^\dagger})$ and 
$q^{i-1}A_{i,i}\equiv q^{i-1} 1_{n_i} \mod q^N$. In particular, there exists $c>0$
such that $A$ and $A_{1,1}^{-1}$ have entries in $\mathcal{O}_{\mathcal{E}}^{1,c}$. 
\begin{proposition} \label{proposition: fancy frob matrix proposition}
	Let $c>c_1>0$ and let $N_0>N>\alpha_d$. There is a Frobenius matrix $A$ of 
	$\eta^* M$ 
	satisfying:

	\begin{enumerate} [label=(\roman*)]
		\item The congruence \eqref{Frobenius structure congruence} holds and
		$A$ has entries in
		$\mathcal{O}_{\mathcal{E}_\infty}^{1,c}\cap 
		\mathcal{O}_{\mathcal{E}_\infty^\dagger}$.\label{matrix condition 0}
		\item (Bottom left) For any $i>j$, the matrix $q^NA_{i,j}$ is divisible 
		by $q^{N_0}$. \label{matrix condition 1}
		\item (First column) For $i>2$, the matrix $q^NA_{i,1}$ has entries in 
		$\mathcal{O}_{\mathcal{E}_\infty}^{1,c_1}$. 
		\label{matrix condition 2}
		\item (First row) For $j\geq 1$, the matrix $A_{1,j}$ has entries in 
		$\mathcal{O}_{\mathcal{E}_\infty}^{(q)}$. \label{matrix condition 4}
		\item  The matrix $A_{2,1}$ has entries $T^{-1} 
		\mathcal{O}_{\mathcal{E}_\infty^{\dagger,-}}$. The
		matrix $A_{i,i}$ lies in $1_{n_i} + qT^{-1} M_{n_i \times 
			n_i}(\mathcal{O}_{\mathcal{E}_\infty^{\dagger,-}})$ for
		$i=1,2$.
		\label{matrix condition 3}
		
	\end{enumerate}
\end{proposition}
\begin{proof}
	The Frobenius matrix \eqref{frobenius matrix 2}
	already satisfies \ref{matrix condition 0}. We prove
	\ref{matrix condition 1}-\ref{matrix condition 3} in two steps. We must ensure 
	that the second step does not undo
	the first step. To this end, we increase $N_0$ so that Lemma \ref{lemma: 
		existence of growth 
		divisibility property} is satisfied. 
	\begin{step}
		\normalfont  Consider the matrix 
		\begin{align*}
		B &= \begin{pmatrix}
		1_{n_1} & 0 & 0&\dots & ~0  \\
		-q^NA_{2,1} & 1_{n_2}&0&\dots & ~0 \\
		-q^N A_{3,1} & -q^N A_{3,2} & 1_{n_3} & \dots& ~0 \\
		\vdots & \vdots & \vdots & \ddots & ~\vdots \\
		-q^NA_{d,1} & -q^NA_{d,2} & -q^N A_{d,3} & \dots  & ~1_{n_d}
		\end{pmatrix}.
		\end{align*}
		By Lemma \ref{dividing by p^s for rc growth}, we know $BA B^{-\sigma}$ has 
		entries in 
		$\mathcal{O}_{\mathcal{E}}^{1,c}$. Also, we have 
		\begin{align*}
		B A B^{-\sigma} &\equiv \begin{pmatrix}
		A_{1,1} & q^N * & q^N*& \dots & q^N * \\
		0 & qA_{2,2} & q^N*& \dots & q^N *\\
		0 & 0 & q^2A_{3,3} & \dots & q^N * \\ 
		\vdots & \vdots &  \vdots & \ddots &\vdots \\
		0 & 0 & 0 & \dots & q^{d} A_{d,d}
		\end{pmatrix} \mod q^{N+1}.
		\end{align*}
		Furthermore, from Lemma \ref{dividing by p^s for rc growth} we
		see that the $(j,1)$-block matrix has entries in 
		$\mathcal{O}_{\mathcal{E}}^{1,q^{-1}c}$
		for $j>2$. 
		After repeating this finitely many times we obtain a Frobenius matrix 
		satisfying \ref{matrix condition 0}-\ref{matrix condition 2}.
	\end{step}
	
	%	\begin{step}
	%		\normalfont  
	%		Consider 
	%		the matrix
	%		\begin{align*}
	%			B&=\begin{pmatrix} 1 & 0 & 0 & 0 &\dots & 0 \\
	%			0 & 1 & 0 & 0 & \dots & 0 \\
	%			-q^{N_0}A_{3,1}A_{1,1}^{-1} & 0 & 1 &0 & \dots & 0 \\ 
	%			\vdots & \vdots & \vdots & \ddots & \vdots \\
	%			-q^{N_0}A_{d,1}A_{1,1}^{-1} & 0 & 0 & 0 & \dots & 1
	%			\end{pmatrix}.
	%		\end{align*}
	%		Note that $BAB^{-\sigma}$ satisfies \ref{matrix condition 0} and 
	%		\ref{matrix condition 1}. For any $j>2$, the $(j,1)$-block 
	%		matrix of $BAB^{-\sigma}$ is $C_j=q^{N_0}A_{j,1}^\sigma 
	%		A_{1,1}^{-\sigma}q^jA_{j,j}$. Write $A_{1,1}^{-1}=1_{d_1} + q^{N}Y$,
	%		where $q^{N}Y$ has entries in $\mathcal{O}_{\mathcal{E}}^{1,c}$. 
	%		This gives
	%		\begin{align*}
	%			C_j&=q^{N_0}A_{j,1}^\sigma q^jA_{j,j}  -q^{N_0}A_{j,1}^\sigma 
	%			q^{N}Y^{\sigma}q^jA_{j,j}.
	%		\end{align*}
	%		By Lemma \ref{lemma: existence of growth divisibility property},
	%		we see that $C_j$ has entries in $\mathcal{O}_{\mathcal{E}}^{1,q^{-1}c}$.
	%		By repeating this step finitely many times we obtain a Frobenius matrix 
	%		satisfying \ref{matrix condition 0}-\ref{matrix condition 2}.
	%	\end{step}
	%	
	
	\begin{step} 	\normalfont
		By the previous step, we may assume \ref{matrix condition 
			0}-\ref{matrix condition 2} hold. 
		Let $m>0$ be large enough so that $A$ has entries in 
		$\mathcal{O}_{\mathcal{E}_\infty}^{\dagger,m}$ (see \S \ref{subsection: aux 
			rings}). Assume that $A$ satisfies properties \ref{matrix 
			condition 4} and \ref{matrix condition 3}
		modulo $q^v$. We will find a matrix $B$ such that:
		\begin{itemize}
			\item $B^\sigma \in M_{n\times 
				n}(\mathcal{O}_{\mathcal{E}_\infty}^{\dagger,m})$ and $B  
			\equiv 1_n \mod q^v$.
			\item $B A B^{-\sigma}$ satisfies properties \ref{matrix condition 
				0}-\ref{matrix condition 2} and properties \ref{matrix condition 4}-\ref{matrix condition 3}
			modulo $q^{v+1}$.
		\end{itemize}
		The proposition will follow by taking the limit. This limit exists
		because
		$\mathcal{O}_{\mathcal{E}_\infty}^{\dagger,m}$ is $p$-adically closed.
		
		For $j>1$ (resp. $j=1$) we 
		write $q^N A_{1,j}=C_j + D_j^\sigma $
		(resp. $A_{1,1} = 1_{d_1} + C_1 + D_1^{\sigma}$), where $C_j$ 
		has 
		entries in $\mathcal{O}_{\mathcal{E}}^{(q)}$. 
		We have $v\geq v_q(D_j)$ for each $j$ by our inductive assumption. 
		Consider
		the matrix
		\begin{align*}
		B_0&= \begin{pmatrix} 
		1_{n_1} -D_1& -D_2  & \dots  & -D_d \\
		0 & 1_{n_2} &  \dots & 0 \\
		\vdots & \vdots & \ddots  & \vdots \\
		0 & 0  & \dots & 1_{n_d}
		\end{pmatrix}.
		\end{align*}
		Note that $B_0 A B_0^{-\sigma}$ satisfies \ref{matrix condition 
			0}-\ref{matrix condition 2} (we are using Lemma \ref{lemma: existence of 
			growth divisibility property} to verify \ref{matrix condition 2}).
		Furthermore, $B_0^{-\sigma}\in M_{n\times 
			n}(\mathcal{O}_{\mathcal{E}_\infty}^{\dagger,m})$, since each
		$D_j^\sigma$ has entries in $\mathcal{O}_{\mathcal{E}_\infty}^{\dagger,m}$.
		The top row of $B_0 A B_0^{-\sigma}$ is equivalent to
		$\begin{bmatrix} 1+C_1 & C_2 & \dots &C_d\end{bmatrix}$
		modulo $q^{v+1}$.

		By the previous paragraph, we may assume 
		$A$ satisfies \ref{matrix condition 0}-\ref{matrix condition 2}
		and \ref{matrix condition 4} modulo $q^{v+1}$.
		Write $q^NA_{2,1} = X_{2,1} + Y_{2,1}$ with $X_{2,1}\in 
		T^{-1} \mathcal{O}_{\mathcal{E}_\infty^{\dagger,-}}$ and $Y_{2,1} \in 
		\mathcal{O}_E[[T]] \otimes \mathcal{O}_{E_{\infty}}$. 
		Similarly, we write $A_{i,i}= X_{i,i} + Y_{i,i}$ with
		$X_{i,i}\in 
		T^{-1} \mathcal{O}_{\mathcal{E}_\infty^{\dagger,-}}$ and $Y_{i,i} \in 
		\mathcal{O}_E[[T]] \otimes \mathcal{O}_{E_{\infty}}$ for $i=1,2$. There exists
		$Z_i$ with entries in 
		$\mathcal{O}_E[[T]] \otimes \mathcal{O}_{E_{\infty}}$ such that 
		$Z_i^\sigma - Z_i = Y_{i,i}$ (note that this is only true after tensoring by 
		$\mathcal{O}_{E_{\infty}}$). Consider the matrix 
		\begin{align*}
		B_1&= \begin{pmatrix}
		1_{n_1} +Z_1 & 0 & 0 & \dots & 0 \\
		-Y_{2,1} & 1_{n_2}+Z_2& 0 & \dots & 0 \\
		0 & 0 & 1_{n_3}&  \dots & 0 \\
		\vdots & \vdots &\vdots & \ddots  & \vdots \\
		0 & 0 & 0 & \dots & 1_{n_d}
		\end{pmatrix}.
		\end{align*}
		The matrix $B_1AB_1^{-\sigma}$ still satisfies \ref{matrix condition 4}
		modulo $q^{v+1}$ and \ref{matrix condition 0}-\ref{matrix condition 2} 
		(again, we use Lemma \ref{lemma: existence of growth divisibility property} 
		to 
		verify \ref{matrix condition 2}). Furthermore,
		we see that $B_1AB_1^{-\sigma}$ satisfies \ref{matrix condition 3} modulo 
		$q^{v+1}$
		by construction. 
	\end{step}
\end{proof}

\begin{remark}
	In the proof Proposition \ref{proposition: fancy frob matrix proposition},
	it was sufficient to remain in 
	$\mathcal{O}_{\mathcal{E}}^{1,c}$ for properties \ref{matrix condition 
		0}-\ref{matrix condition 4}. It is only necessary to go to the geometric 
		fiber 
	$\eta^* M$ for property \ref{matrix condition 3}.
\end{remark}

\subsection{The Frobenius structure of the unit-root subcrystal}
We will continue with the setup from the beginning of \S \ref{subsubsection: integer 
	slopes}, with 
the additional assumption that $M^{unit}$ has rank one (see 
\S \ref{subsection: slope filtrations}). This subsection is dedicated to proving the 
following proposition:
\begin{proposition}\label{min frob prop for unit root}
	There exists a maximal Frobenius $\lambda \in 
	\mathcal{O}_{\mathcal{E}_\infty}$ of $\eta^* M^{unit}$ such that the following is 
	satisfied: for any $c_1>0$,
	there exists $\lambda_{c_1} \in \mathcal{M}$ with
	\begin{align*}
	\lambda &\equiv \lambda_{c_1} \mod \mathcal{O}_{\mathcal{E}_\infty }^{1,{c_1}}.
	\end{align*}
\end{proposition}

Let $N_0$ be sufficiently large so that the results of \S 
\ref{subsubsection approximating frobenius equations} hold. Let $A$ be a Frobenius 
matrix of $\eta^* M$ satisfying the properties in 
Proposition \ref{proposition: fancy frob matrix proposition}
and write $A$ as a block matrix as in \eqref{frobenius matrix 2}. Let $\mathbf{e} 
=(e_1, \dots, e_n)$ be the basis of $\eta^* M$
such that $\varphi(\mathbf{e}^T)=A\mathbf{e}^T$ and let $u \in \eta^*M^{unit}$.
After normalizing $u$ we have
\begin{align*}
u &=  e_1 + \epsilon_2 e_2 + \dots + \epsilon_n e_n,
\end{align*}
where $\epsilon_i \in \mathcal{E}_\infty$.
The Frobenius structure of $\eta^* M^{unit}$ is given by $\lambda\in 
\mathcal{E}_\infty$ satisfying
$\varphi(u)=\lambda u$. 
Write $a_{i,j}$ for the $(i,j)$-th entry of $A$ and let $b_{i,j} = 
\frac{a_{i,j}}{a_{1,1}}$. We obtain the following equations:
\begin{align}
\lambda &= a_{1,1}+ a_{1,2}\epsilon_2^\sigma+ \dots + 
a_{1,n}\epsilon_n^\sigma, \label{main Frob eq: 1} \\
\epsilon_i &= b_{i,1}+ b_{i,2}\epsilon_2^\sigma + \dots 
b_{i,n}\epsilon_n^\sigma-(b_{1,2}\epsilon_2^\sigma\epsilon_i+ \dots + 
b_{1,n}\epsilon_n^\sigma\epsilon_i). \notag
\end{align}
By Lemma \ref{proposition: fancy frob matrix proposition}-\ref{matrix condition 4},
we know $\lambda \in \mathcal{O}_{\mathcal{E}_\infty}^{(q)}$ and by
Lemma \ref{lemma: non divisble are maximal} we know $\lambda$ is maximal. 
Define the matrices
\[ B 
= \begin{bmatrix} b_{2,2} & \dots & b_{2,n} \\ 
\vdots & \ddots & \vdots \\
b_{n,2} & \dots & b_{n,n} \end{bmatrix} ~\text{ and}~~ \mathbf{x}=\begin{bmatrix} x_2 
& \dots & 
x_n \end{bmatrix},\]
where $x_i = b_{i,1} 
- \sum\limits_{j=2}^n b_{1,j}\epsilon_j^\sigma \epsilon_i$.
It will be convenient to write $B$ as the block matrix 
$\begin{bmatrix}
B_{1,1} & B_{1,2} \\
B_{2,1} & B_{2,2},
\end{bmatrix} $, where
$B_{1,1} = a_{1,1}^{-1} A_{2,2}$.
The vector $\mathbf{\epsilon}=\begin{bmatrix} \epsilon_2 & \dots & 
\epsilon_n \end{bmatrix}$
satisfies the recursive equation
\begin{align} \label{unit-root recursive equation}
\mathbf{\epsilon}^T &= B\mathbf{\epsilon}^{\sigma T} + \mathbf{x}^T.
\end{align}
From Proposition \ref{proposition: fancy frob matrix 
	proposition} and Lemma \ref{lemma: existence of growth divisibility property} 
we may deduce the following lemma about $B$:

\begin{lemma}
	\label{lemma: specific properties of Frobenius matrix}
	We have the following:
	\begin{enumerate}
		\item The entries of $B$ lie in $\mathcal{O}_{\mathcal{E}_\infty}^{1,c} \cap 
		\mathcal{O}_{\mathcal{E}_\infty^\dagger}$.
		\item The block matrix $B_{2,1}$ has entries in 
		$\mathcal{O}_{\mathcal{E}_\infty}^{1,c} \cap 
		q^{N_0}\mathcal{O}_{\mathcal{E}_\infty}$.
		\item For $i>n_2+1$ we have $b_{i,1} \in 
		\mathcal{O}_{\mathcal{E}_\infty}^{1,c_1} \cap 
		q^{N_0}\mathcal{O}_{\mathcal{E}_\infty}$. 
		\item The matrix $B_{1,1}$ is of the form $q1_{n_2} + q^2C$, where
		$q^2C $ has entries in $T^{-1}
		\mathcal{O}_{\mathcal{E}_{\infty}^{\dagger,-}}$.
	\end{enumerate}
\end{lemma}

\begin{lemma} \label{unit-root coordinates have the correct growth}
	For each $i$ we have $\epsilon_i \in \mathcal{O}_{\mathcal{E}_\infty}^{1,c}\cap 
	q^{N_0} 
	\mathcal{O}_{\mathcal{E}_\infty}$.
\end{lemma}

\begin{proof}
	We first prove that $v_q(\mathbf{\epsilon})\geq N_0$. Let 
	$k=v_q(\mathbf{\epsilon})$. By the definition of $\mathbf{x}$
	and the fact that $q^{N_0}|b_{i,1}$ we
	see $v_q(\mathbf{x})\geq \min(k+1,N_0)$. Then from \eqref{unit-root 
		recursive equation}
	we see that $v_q(\mathbf{\epsilon})\geq \min(k+1,N_0)$. Thus, 
	$v_q(\mathbf{\epsilon}) \geq N_0$.
	Next, we show that $\mathbf{\epsilon}$ has entries in 
	$\mathcal{O}_{\mathcal{E}_\infty}^{1,c} + q^k\mathcal{O}_{\mathcal{E}_\infty}$
	for every $k$. For $k=N_0$ this is immediate. Let 
	$k\geq N_0$ and assume  
	$\mathbf{\epsilon}$ has entries in $\mathcal{O}_{\mathcal{E}_\infty}^{1,c} + 
	q^k\mathcal{O}_{\mathcal{E}_\infty}$. Then $\epsilon_i=u_i + q^kv_i$, where $u_i 
	\in 
	\mathcal{O}_{\mathcal{E}_\infty}^{1,c}$. As $q|B$, we know
	\begin{align}
	B\mathbf{\epsilon}^{\sigma T} &\equiv B \begin{bmatrix} u_2^\sigma & \dots & 
	u_n^\sigma 
	\end{bmatrix}^{ T} \mod q^{k+1}.\label{unitroot growth eq1}
	\end{align}
	By Lemma 
	\ref{dividing by p^s for rc growth}, the right side of \eqref{unitroot growth 
		eq1} is contained in 
	$\mathcal{O}_{\mathcal{E}_\infty}^{1,c}$.
	Thus, $B\mathbf{\epsilon}^{\sigma T}$ has entries  in 
	$\mathcal{O}_{\mathcal{E}_\infty}^{1,c} + 
	q^{k+1}\mathcal{O}_{\mathcal{E}_\infty}$. Similarly, 
	observe $\mathbf{x}$ has entries in $\mathcal{O}_{\mathcal{E}_\infty}^{1,c} 
	+ 
	q^{k+1}\mathcal{O}_{\mathcal{E}_\infty}$. By \eqref{unit-root 
		recursive 
		equation}, we know $\mathbf{\epsilon}$ has entries in 
	$\mathcal{O}_{\mathcal{E}_\infty}^{1,c} 
	+ q^{k+1} \mathcal{O}_{\mathcal{E}_\infty}$.
\end{proof}

\begin{lemma} \label{higher slope coordinates do not matter lemma}
	Let $B_0$ denote the $(n-1)\times (n-1)$ matrix $\begin{bmatrix} B_{1,1} & 0 \\ 0 
	& 0 \end{bmatrix}$ and let $\mathbf{x_0}=\begin{bmatrix} 
	b_{2,1} 
	& \dots & b_{n_2+1, 1} &  0 & \dots & 0 	\end{bmatrix}$. Then
	\begin{align*}
	R(B,\mathbf{x}^T) &\equiv R(B_0, \mathbf{x_0}^T) \mod 
	\mathcal{O}_{\mathcal{E}_\infty}^{1,c_1}\cap q^{N_0} 
	\mathcal{O}_{\mathcal{E}_\infty}.
	\end{align*}		
	
\end{lemma}

\begin{proof}
	First, note that $b_{i,j}\epsilon_j^\sigma\epsilon_i \in 
	\mathcal{O}_{\mathcal{E}_\infty}^{1,c_1} \cap 
	q^{N_0}\mathcal{O}_{\mathcal{E}_\infty}$. This follows from Lemma \ref{lemma: 
	existence 
		of growth divisibility property}, Lemma \ref{lemma: specific properties of 
		Frobenius 
		matrix}, and Lemma \ref{unit-root coordinates have the 
		correct growth}. Furthermore, for $i>n_2+1$ we know from Lemma \ref{lemma: 
		specific properties of Frobenius matrix} that $b_{i,1} \in 
	\mathcal{O}_{\mathcal{E}_\infty}^{1,c_1} 
	\cap 
	q^{N_0}\mathcal{O}_{\mathcal{E}_\infty}$. This means 
	$\mathbf{x}-\mathbf{x_0}$ has entries in 
	$\mathcal{O}_{\mathcal{E}_\infty}^{1,c_1} 
	\cap 
	q^{N_0}\mathcal{O}_{\mathcal{E}_\infty}$. By Lemma \ref{lemma: effect of 
	recursive on 
		highly divisible r,d things} we have
	\begin{align*}
	R(B,\mathbf{x})&\equiv R(B,\mathbf{x}_0) \mod 
	\mathcal{O}_{\mathcal{E}_\infty}^{1,c_1}\cap q^{N_0} 
	\mathcal{O}_{\mathcal{E}_\infty}.
	\end{align*}
	Let $B_1=\begin{bmatrix} 0 & B_{1,2} \\ B_{2,1} & B_{2,2} \end{bmatrix}$
	and set $S_m =  R(B_0; 
	\underbrace{B_1,\dots,B_1}_{m \text{ times}}, \mathbf{x}_0)$. 
	By Lemma \ref{addition formula}, it suffices to prove $S_m$ has entries in 
	$\mathcal{O}_{\mathcal{E}_\infty}^{1,c_1} \cap 
	q^{N_0}\mathcal{O}_{\mathcal{E}_\infty}$ for $m\geq 1$. Write
	$S_0=\begin{bmatrix}z_2 & \dots & z_n 
	\end{bmatrix}$. By Corollary \ref{rc growth recursive formula} we know $z_i \in 
	\mathcal{O}_{\mathcal{E}_\infty}^{1,c}\cap 
	q^{N_0}\mathcal{O}_{\mathcal{E}_\infty}$. Also, $z_i=0$ for $i>n_2+1$.
	Thus, by Lemma \ref{lemma: existence of growth 
		divisibility property} and Lemma \ref{lemma: specific properties of Frobenius 
		matrix}, the entries of $B_1 S_0$
	are contained in $\mathcal{O}_{\mathcal{E}_\infty}^{1,c_1} \cap 
	q^{N_0}\mathcal{O}_{\mathcal{E}_\infty}$. The result follows from Lemma 
	\ref{lemma: existence of growth 
		divisibility property} and Lemma \ref{lemma: effect of 
		recursive on highly divisible r,d things}.
\end{proof}

\begin{proof}
	(Of Proposition \ref{min frob prop for unit root})
	Let $\begin{bmatrix}
	y_2 & \dots & y_n
	\end{bmatrix}^T=R(B_0,\mathbf{x_0}^T)$ and note that $y_i=0$ for $i>n_2+1$. 
	By \eqref{main Frob eq: 1}, Lemma \ref{lemma: existence of growth divisibility 
		property} 
	and Lemma \ref{higher slope coordinates do not matter lemma} we see that 
	\begin{align} \label{unit-root approximation 1}
	\lambda &\equiv a_{1,1} + a_{1,2}y_2^\sigma + \dots + 
	a_{1,n_2+1}y_{n_2+1}^\sigma \mod \mathcal{O}_{\mathcal{E}_\infty}^{1,c_1}\cap 
	q^{N_0} 
	\mathcal{O}_{\mathcal{E}_\infty}.
	\end{align}
	Let $\mathbf{y_0}=\begin{bmatrix}
	y_2 & \dots & y_{n_2+1}
	\end{bmatrix}^T$ and note that $\mathbf{y_0} =R(B_{1,1};\begin{bmatrix} b_{2,1} & 
	\dots & 
	b_{n_2+1,1} \end{bmatrix}^T)$.
	By Lemma \ref{lemma: specific properties of Frobenius matrix}, the hypothesis 
	of Proposition \ref{main 
		$r$-recursive Prop} is satisfied. Therefore, there exist $a>0$
	and a column matrix $\mathbf{z}=[z_2,\dots,z_{n_2+1}]^T$ with entries in 
	$\mathcal{N}_a$ (see \S \ref{subsection: spaces of recursive frobenius solutions} 
	for the definition of this space) such that 
	\begin{align*}
	\mathbf{y}_0 \equiv  \mathbf{z} \mod \mathcal{O}_{\mathcal{E}_\infty}^{1,c_1}\cap 
	q^{N_0} 
	\mathcal{O}_{\mathcal{E}_\infty}.
	\end{align*}
	By Lemma \ref{lemma: break up overconvergent things}, after increasing $a$ we may 
	assume that $a_{1,i}=u_i + v_i$, where $u_i \in 
	\mathcal{O}_{E_a}((T))$ and $v_i \in \mathcal{O}_{\mathcal{E}_\infty}^{1,c_1}\cap 
	q^{N_0} 
	\mathcal{O}_{\mathcal{E}_\infty}$. Then we set
	\begin{align*}
	\lambda_{c_1} &= u_1 + u_2z_2^\sigma + \dots + u_{n_2+1}z_{n_2+1}^\sigma,
	\end{align*}
	which is contained in $\mathcal{M}_a$ (see \S \ref{subsection: stable growth for 
	solutions of Frobenius equations} for the definition of this space).
	From Lemma \ref{lemma: existence of growth divisibility property} we have
	\begin{align*}
	\lambda &\equiv \lambda_{c_1} \mod \mathcal{O}_{\mathcal{E}_\infty}^{1,c_1}\cap 
	q^{N_0} 
	\mathcal{O}_{\mathcal{E}_\infty},
	\end{align*} 
	which proves the proposition.
\end{proof}

\section{Main results}
\label{Section: main results}

\subsection{Local results}
\label{subsection: main slope filtration and log decay results}
\begin{theorem} \label{diffences of slopes give log-decay}
	Adopt the notation from \S \ref{subsection: local setup}. Let 
	$r_i=\frac{1}{\alpha_{i+1}-\alpha_i}$. Then $M_i$ has $r_i$-log-decay.
\end{theorem}
\begin{proof}
	We first show that $M_i$ has $r_i$-log-decay for $T$. The general result will 
	follow
	by the fact that the slope filtration is functorial. Note that $M_i$ has 
	$r_i$-log-decay for $T$ if and only if $\det(M_i)$ has $r_i$-log-decay for $T$.
	This follows from a standard exterior power trick (\cite[Theorem 
	2.4.2]{katz-slope_filtration} or
	\cite[proof of Proposition 6.1]{JKM-higher_dim}).
	Thus, by replacing $M$ with $\wedge^{rank(M_i)}M$, it suffices to prove the
	result when $i=1$ under the assumption that $M_1$ has rank one. 
	Let $A$ be a Frobenius matrix of $M$ and let $C$ be the corresponding connection 
	matrix.
	By Lemma \ref{lemma: log growth and slope filtrations}, there exists
	$B \in GL_n(\mathcal{O}_{\mathcal{E}^{r_i}})$ such that
	\begin{align} \label{Frob matrix after change of basis}
	B A B^{-\sigma} &= \begin{pmatrix} X & Y \\ 0 & Z \end{pmatrix}.
	\end{align}
	The connection matrix after this change of basis is 
	\begin{align*}
	B^{-1}(\delta_T(B) + CB) &= \begin{pmatrix} Q & R \\ S & U \end{pmatrix},
	\end{align*}
	whose entries are contained in $\mathcal{O}_{\mathcal{E}^{r_i}}$.
	From \eqref{Frob matrix after change of basis} we see that
	$M\otimes \mathcal{O}_{\mathcal{E}^{r_i}}$ has a sub-$\sigma$-module
	$M_1'$ of rank one. We will prove
	that $S=0$. This will imply that $M_1'$ is fixed by $\nabla$,
	which means $M_1'$ is a $(\sigma,\nabla)$-module over $\mathcal{E}^{r_i}$. The 
	uniqueness of the slope filtration
	will imply $M_1'=M_1$. 
	
	The 
	compatibility between
	the Frobenius and the connection gives the relation
	\begin{align*}
	SX &= \frac{dT^\sigma}{dT} ZS^\sigma.
	\end{align*}
	Then $v_p(SX)=v_p(S)$,
	since $X$ is a $p$-adic unit. Also, note that $v_p(\frac{dT^\sigma}{dT}Z)>0$ and 
	$v_p(S)=v_p(S^\sigma)$.
	In particular, if $v_p(S)<\infty$ then 
	$v_p(\frac{dT^\sigma}{dT}ZS^\sigma)>v_p(S)$.
	It follows that $v_p(S)=\infty$.
\end{proof}

\begin{remark}
	The proof of Theorem \ref{diffences of slopes give log-decay} uses an exterior power trick to reduce to the case of $i=1$ and $M_1$ has rank one. It is natural to ask if the same trick can be applied to Conjecture \ref{log-decay and slope filtrations}. The answer is no (at least without some additional work). 
	An issue arises
	if $\wedge^{rank(M_i)} M$ is not irreducible.
	It could happen that $\wedge^{rank(M_i)} M$ has an irreducible subobject $N$ such that $\det(M_i)$ is the first step
	in the slope filtration for $N$. It is not clear that the second smallest slope of $N$ will be the second smallest
	slope of $\wedge^{rank(M_i)} M$.
\end{remark}

\begin{corollary}  \label{local Drinfeld-Kedlaya}
	Keep the notation from Theorem \ref{diffences of slopes give log-decay}. If 
	$\alpha_{i+1}-\alpha_i>1$, then $M_i$ is overconvergent.
\end{corollary}
\begin{proof}
	Assume that $\alpha_{i+1}-\alpha_i>1$. Let $g \in \Theta$. From Theorem 
	\ref{diffences of slopes give log-decay}
	we know that for $\det(M_i^g)$ has $r_i$-log-decay, where $r_i<1$. By Corollary 
	\ref{corollary: small log decay}, we see that 
	$\det(M_i)$ is overconvergent.  The fully faithfulness of $\iota^\dagger$ (see 
	\cite{Kedlaya-fully_faithful}) implies $\det(M_i)$ is a 
	subobject of $\wedge^{rank(M_i)} M$ in $\Fisoc^\dagger(\Spec(F))\otimes L$. The 
	result follows
	from the same exterior product argument used in Theorem \ref{diffences of slopes 
	give log-decay}.
\end{proof}

\begin{corollary} \label{log-decay conjecture holds for ordinary}
	Let $M$ be an irreducible object of $\Fisoc^\dagger(\Spec(F))$
	with integral slopes. Then Conjecture \ref{log-decay and slope filtrations} holds.
\end{corollary}
\begin{proof}
	We know that $M_i$ has $1$-log-decay by Theorem \ref{diffences of slopes give 
		log-decay}.
	If there exists $r<1$ such that $M_i$ has $r$-log-decay, then $\det(M_i)$
	also has $r$-log-decay. From Corollary \ref{corollary: small log decay} 
	we find that $\det(M_i)$ is overconvergent. Then Kedlaya's fully faithfulness 
	theorem tells us that $M$ is not irreducible. 
\end{proof}

\begin{theorem} \label{main local result}
	Adopt the notation from Theorem \ref{diffences of slopes give log-decay} with
	the additional assumption that $M$ is 
	irreducible. Let $\rho_i$ denote a 
	Galois representation associated to $\det(M_i)$ (see \S \ref{paragraph:galois 
		reps for isoclinic}). The 
	following hold:
	\begin{enumerate}
		\item $\rho_i$ has $r_i$-bounded monodromy.
		\item If $M$ has integer slopes, then $\rho_i$ has stable 
		monodromy. 
	\end{enumerate}
\end{theorem}

\begin{proof}	
	The first statement follows from Corollary \ref{corollary: small log decay} and 
	Proposition \ref{diffences of slopes give log-decay}. Assume the slopes of 
	$M$ are integers.
	It suffices 
	to prove this result for $i=1$ under the assumption that $M_1$ is unit-root and 
	has rank one. After replacing $F$ with a finite ramified extension we may assume 
	that $\rho_1$ is harshly ramified (see Corollary \ref{corollary: harsh 
		ramification base change}). We may also assume Proposition 
	\ref{min frob prop for unit root} is satisfied. These base changes do not change 
	the result by Corollary \ref{corollary: pseudostable monodromy after base 
		change}. Since $\rho_1$ is harshly 
	ramified, there exists $c_1$ such that $s_k(\rho_1) >c_1p^{k}$ for all $k\geq 0$. 
	By Proposition \ref{min frob prop for unit root}, for each $g\in \Theta$ 
	there exists a maximal Frobenius $\lambda_g$ of $M_1^g$ with
	$\lambda_g=\lambda_{g,c_1} + z_{g,c_1}$, where $\lambda_{g,c_1} \in \mathcal{M}$ 
	and $z_{g,c_1}
	\in \mathcal{O}_{\mathcal{E}_\infty}^{1,c_1}$. From Corollary \ref{corollary: 
		harshly ramified monodromy}, we know
	\begin{align*}
	s_k(\rho_1) &= \max_{g \in \Theta} \{-w_k(\lambda_{g,c_1})\}
	\end{align*} 
	for all $k$. The theorem follows from Proposition \ref{proposition: main 
		recursive growth result}.
\end{proof}

\begin{corollary} \label{main local proposition}
	Let $\psi: V \to \Spec (F)$ be a smooth proper morphism. Then
	$M=R^i \psi_* \mathcal{O}_{V,cris}$ is an object of $\Fisoc(\Spec(F))$.
	Assume that $R^i \psi_* \underline{\Q_p^{et}}$ and let $\rho$ be the Galois 
	representation corresponding to the lisse \'etale sheaf  
	$\det(R^i \psi_* \underline{\Q_p^{et}})$. 
	\begin{enumerate}
		\item Let $\alpha$ denote the first nonzero slope of $M$ and let 
		$r=\frac{1}{\alpha}$. Then $\rho$ has $r$-bounded monodromy.
		\item Assume $M$ has integer slopes. Then either $\rho$ has finite monodromy 
		or $\rho$ has stable monodromy. 
	\end{enumerate}
\end{corollary}
\begin{proof}
	By Theorem \ref{theorem: Berthelot-Kedlaya} we know $R^i \psi_* \mathcal{O}_{V,cris}$ is an object of 
	$\Fisoc^\dagger(\Spec(F))$. The corollary follows 
	from
	Theorem \ref{main local result}, as $M^{unit}$ corresponds to $R^i \psi_* 
	\Q_p^{et}$.
\end{proof}

\subsection{Global results}
\label{subsection: main monodromy results}

\begin{theorem} (Drinfeld-Kedlaya) \label{global Drinfeld-Kedlaya}
	Let $C$ be a smooth curve and let $M$ be an irreducible 
	object of $\Fisoc(C)\otimes L$ or $\Fisoc^\dagger(C)\otimes L$. The differences 
	between the 
	consecutive generic slopes
	of $M$ are bounded by one. 
\end{theorem}
\begin{proof}
	
	Assume that $M$ has two consecutive generic slopes whose difference is greater
	than one. By taking exterior powers and twisting we may assume that
	the first two generic slopes of $M$ are $0$ and $\alpha>1$. In particular, 
	$(1,0)$ 
	is a vertex on
	the generic Newton polygon of $M$. For any map $x \to C$ we let $M|_x$ denote the pullback of $M$ to $x$. 
	We will show that if $x \to C$ corresponds to a closed point of $C$, then $NP(M|_x)$ contains the vertex 
	$(1,0)$. 
	Let $\mathcal{O}_F=\widehat{\mathcal{O}_{C,x}}$ and let $F$
	be the fraction field of $\mathcal{O}_F$. Let $S=\Spec (\mathcal{O}_F)$ and
	let $T=\Spec (F)$. Then $M|_T$ is overconvergent. 
	By Corollary \ref{local Drinfeld-Kedlaya} we know that
	$M^{unit}|_{T}$ is overconvergent. In particular, there exists a finite 
	totally ramified extension
	$F'$ of $F$ such that $M^{unit}|_{\Spec(F')}$ extends to 
	$\Spec(\mathcal{O}_{F'})$.
	This means that if $x' \in \Spec(\mathcal{O}_{F'})$ is the special point,
	then $NP(M|_{x'})$ contains the vertex $(1,0)$. However,
	since $F'$ is totally ramified over $F$ we have $M|_{x'} \cong M|_{x}$,
	which means $NP(M|_{x})$ contains the vertex $(1,0)$.
	
	By the preceding paragraph, we know that for each $x \in C$, closed or generic, the Newton polygon
	$NP(M|_x)$ contains the point $(0,1)$. Then if $M$ is in $\Fisoc(C)\otimes L$, Katz' slope filtration theorem (see
	\cite[Theorem 2.4.2]{katz-slope_filtration})
	implies that $M$ has a subobject $M^{unit}$, and thus $M$ is reducible in 
	$\Fisoc(C)\otimes L$. Now, let $M$ be an object of $\Fisoc^\dagger(C)\otimes L$. 
	Let $X$ be the smooth compactification of $C$ and let
	$y \in X-C$ be a point at infinity. We let $K$
	be the fraction field of $\widehat{\mathcal{O}_{X,y}}$. Then $M|_{\Spec(K)}$
	is overconvergent and by Corollary \ref{local Drinfeld-Kedlaya}
	we know that $M^{unit}|_{\Spec (K)}$ is also overconvergent. It follows that
	$M^{unit}$ is overconvergent. Then by Kedlaya's fully faithfulness theorem
	we know that $M^{unit}$ is a subobject of $M$ in $\Fisoc^\dagger(C) \otimes L$.
	
\end{proof}

\begin{remark} \label{remark: drinfeld-kedlaya}
	Our proof of Theorem \ref{global Drinfeld-Kedlaya} is somewhat perpendicular
	to the work of Drinfeld and Kedlaya. In \cite{Drinfeld-Kedlaya},
	they shrink $C$ until the Newton polygons is the same at each point of $C$. 
	They then prove that certain $Ext$
	groups vanish, using ideas that trace back to Kedlaya's thesis
	(see \cite[5.2.1]{Kedlaya-thesis}). From more advanced faithfulness results of 
	Kedlaya, Shiho,
	and de Jong, it follows from general nonsense that $M$ decomposes
	into the direct sum of two $\mathbf{F}$-isocrystals. In contrast,
	we use Corollary \ref{corollary: small log decay} and Theorem \ref{diffences of 
		slopes give log-decay} to prove the Newton polygon at each $x 
	\in C$
	agrees with the generic Newton polygon.
\end{remark}

\begin{theorem} \label{main global theorem}
	Let $U$ be a smooth curve over $k$ and let $C$ be its smooth
	compactification. Let $M$ be an irreducible object of $\Fisoc^\dagger(U)$. After 
	replacing $U$ with a dense open subset,
	there is a slope filtration $0=M_0 \subset M_1 \subset \dots \subset M_d = 
	\iota^\dagger(M)$, where each graded piece $M_{i+1}/M_i$ is isoclinic of slope 
	$\alpha_i$. Let $\rho_i$ be a $p$-adic character of 
	$\pi_1^{et}(U)$ corresponding to $\det(M_i)$.
	\begin{enumerate}
		\item $\rho_i$ has $r_i$-bounded genus growth, where 
		$r_i=\frac{1}{\alpha_{i+1}-\alpha_i}$.
		\item If the slopes of $M$ are all integers, then $\rho_i$ is genus 
		psuedo-stable.
	\end{enumerate}
\end{theorem}

\begin{proof}
	Let $D=C-U$. For each $x\in D$, let $F_x=\Frac{\widehat{\mathcal{O}}_{C,x}}$ be 
	the 
	fraction field of the completed local ring at $x$. Let $\rho_{x,i}$ denote the 
	representation induced by pulling back $\rho_i$ along $\eta_x: \Spec(F_x) \to U$. 
	Then $\rho_{x,i}$ corresponds to $\eta_x^*(\det(M_i))$. By the Riemann-Hurwitz 
	formula (see \cite{Oort-RH_formula}) we know that $\rho_i$ is genus psuedo-stable 
	if
	$\rho_{x,i}$ has stable monodromy for each $x \in D$. Also, $\rho_i$ has 
	$r_i$-bounded genus growth 
	if $\rho_{x,i}$ has $r_i$-bounded monodromy for each $x \in D$. The result 
	follows 
	from Theorem \ref{main 
		local result}.
\end{proof}

\begin{corollary} \label{main global proposition}
	Let $U$ be a smooth curve and let $\psi: V \to U$ be a smooth proper morphism.
	Let
	$M=R^i \psi_* \mathcal{O}_{V,cris}$ and assume that the generic
	slopes of $M$ are integers. Then the $p$-adic character of 
	$\pi_1^{et}(U)$ corresponding to $\det(R^i \psi_* \Q_p^{et})$ is genus 
	psuedo-stable.
\end{corollary}
\begin{proof}
	Let $X$ be the smooth compactificaton of $U$. By Theorem \ref{theorem: Berthelot-Kedlaya} applied to each
	$x \in X-U$, we know
	that $M$ is an object of $\Fisoc^\dagger (U)$. The $p$-adic character of 
	$\pi_1^{et}(U)$ corresponding to $\det(R^i \psi_* \Q_p^{et})$ is the same as 
	$\rho_1$ 
	from Theorem \ref{main global theorem}.
\end{proof}

\appendix
\section{A local Berthelot's conjecture for constant sheafs}
\label{section: berthelot's conjecture}
Berthelot's conjecture states that the higher direct images of an overconvergent $F$-isocrystal along
a smooth proper morphism are again overconvergent $F$-isocrystals. 
In this appendix we prove a local version of this conjecture for the higher direct images of the constant sheaf. This is originally due to Kedlaya in \cite[Chapter 7]{Kedlaya-thesis}, as an application of a fully faithfulness result. However, to 
the best of our knowledge there is no published proof. We follow closely the proof in \cite[Chapter 7]{Kedlaya-thesis}.
\begin{theorem} \label{theorem: Berthelot-Kedlaya}
	Let $\psi: V \to \Spec (F)$ be a smooth proper morphism. Then
	$M=R^i \psi_* \mathcal{O}_{V,cris}\otimes \Q_p$, which a priori is an object of $\Fisoc(\Spec(F))$, is overconvergent.
\end{theorem} 
\noindent We begin with two lemmas.
\begin{lemma} \label{lemma: log-crystals}
	Let $\overline{\psi}: \overline{V} \to \Spec(\mathfrak{K}[[T]])$ be a proper morphism whose special fiber is reduced with strict normal crossings. Let $\psi$ be the generic fiber of of $\overline{\psi}$.
	Let $A$ be the standard fine log structure associated to $\Spec(\mathfrak{K}[[T]])$ and let $B$ be the 
	fine log structure associated to the special divisor of $\overline{V}$. In particular, $\overline{\psi}:(\overline{V},B) \to (\Spec(\mathfrak{K}[[T]]), A)$
	is a smooth map of log schemes. Then $R^i\psi_*(\mathcal{O}_{V,cris})\otimes \Q_p$ is
	an object of $\Fisoc^\dagger(\Spec(F))$.
\end{lemma}

\begin{proof}
	First note that $R^i\psi_*(\mathcal{O}_{V,cris})\otimes \Q_p$ and $R^i\overline{\psi}_*(\mathcal{O}_{\overline{V},cris}^{log})\otimes \Q_p$ give rise to the same element of
	$\Fisoc(\Spec(F))$. That is, the restriction of $R^i\overline{\psi}_*(\mathcal{O}_{\overline{V},cris}^{log})\otimes \Q_p$
	to the generic point of $\Spec(\mathfrak{K}[[T]])$ is the same as $R^i\psi_*(\mathcal{O}_{V,cris})\otimes \Q_p$.
	Thus, it suffices to prove the result for $R^i\overline{\psi}_*(\mathcal{O}_{\overline{V},cris}^{log})\otimes \Q_p$.
	First, we use $\mathcal{O}_L[[T]]$ as a test object in the log-crystalline site over $\mathfrak{K}[[T]]$. Thus,
	$R^i\overline{\psi}_*(\mathcal{O}_{\overline{V},cris}^{log})\otimes \Q_p$ gives rise to a free $\mathcal{O}_L[[T]]\otimes L$-module with a connection and semi-linear Frobenius map. From \cite[Proposition (2.24)]{Hyodo-Kato} we deduce that the Frobenius structure must be an isomorphism. Since $\mathcal{O}_L[[T]]\otimes L \subset \mathcal{E}^\dagger$, we see that the $(\sigma,\nabla)$-module over $\mathcal{E}$ associated to 
	$R^i\psi_*(\mathcal{O}_{V,cris})\otimes \Q_p$ descends to $\mathcal{E}^\dagger$. 
\end{proof}

\begin{lemma}
	\label{lemma: injection on cohomology}
	Let $\psi:X \to \Spec(F)$ and $\eta:Y \to \Spec(F)$ smooth, irreducible, proper varieties over $\Spec(F)$. Let $f:X\to Y$ be a surjective morphism. 
	Then the induced map $f^*: R^i\eta_*(\mathcal{O}_{Y,cris}) \to R^i\psi_*(\mathcal{O}_{X,cris})$ is injective
	and there exists a projector mapping $R^i\psi_*(\mathcal{O}_{X,cris})$ to the image of $f$.
\end{lemma}
\begin{proof}
	Crystalline cohomology for varieties over $\Spec(F)$ is a Weil cohomology theory (see e.g. \cite{Berthelot-crystalline-cohomology-book}). The injectivity then follows from \cite[Proposition 1.2.4]{Kleiman-alg_cycles_weil_conj}. Furthermore, from the proof of \cite[Proposition 1.2.4]{Kleiman-alg_cycles_weil_conj} we see that $f_*f^*$ is injective. The projector is then $\frac{1}{d}f^*f_*$, where $d$ is the generic degree of $f$. 
\end{proof}
\begin{corollary}
\label{corollary: OC using projector}
Let $X$ and $Y$ be as in Lemma \ref{lemma: injection on cohomology}. If $R^i\psi_*(\mathcal{O}_{X,cris})$ is an object of
$\Fisoc^\dagger(\Spec(F))$, then $R^i\eta_*(\mathcal{O}_{Y,cris})$ is an object of
$\Fisoc^\dagger(\Spec(F))$.
\end{corollary}
\begin{proof}
	Recall from \S \ref{subsection: fisocrystal section} that $\iota^\dagger$ is fully faithful. This means that
	the projector from Lemma \ref{lemma: injection on cohomology}, which is a morphism in $\Fisoc(\Spec(F))$, must be
	a morphism in $\Fisoc^\dagger(\Spec(F))$. The corollary follows.
\end{proof}

\begin{lemma} \label{lemma: finite field extension and OC}
	Let $\phi:\Spec(F_0)\to \Spec(F)$ be a finite field extension. Let $M$ be an object of $\Fisoc(\Spec(F))$. If
	$\phi^*M$ is overconvergent, then $M$ is overconvergent.
\end{lemma}
\begin{proof}
	For separable extensions this follows from the fully faithfulness of $\iota^\dagger$ and
	the fact that $M$ is a direct summand of $\phi_*\phi^*M$. 
	For the inseparable case, it suffices to prove the lemma for the degree $p$ case.
	Thus, we may assume $F=\mathfrak{K}((T))$ and $F_0 = \mathfrak{K}((T^{\frac{1}{p}}))$. 
	We may view $\mathcal{E}_{F,L}$ as a subfield of $\mathcal{E}_{F_0,L}$ and note that
	$\sigma$ extends to a Frobenius endomorphism of $\mathcal{E}_{F_0,L}$ sending $T^{\frac{1}{p}}$ to $T$. 
	Let $M$ be a $(\sigma,\nabla)$-module over $\mathcal{E}_{F,L}$ and let
	$A$ be a Frobenius matrix of $M$ corresponding to a basis $\mathbf{e}$ (i.e. $\varphi(\mathbf{e})=A\mathbf{e}$).
	Assume that $M \otimes \mathcal{E}_{F_0,L}$ descends to a $(\sigma,\nabla)$-module
	over $\mathcal{E}_{F_0,L}^\dagger$. This means that for some $B$ with
	entries in $\mathcal{E}_{F_0,L}$, the connection matrix $C_0$ and the Frobenius matrix $A_0=BAB^{-\sigma}$ for the basis $B\mathbf{e}$
	are overconvergent (i.e. have entries in $\mathcal{E}_{F_0,L}^\dagger$). Note that
	$B^\sigma$ has entries in $\mathcal{E}_{F,L}^\dagger$ since $F_0$ is a degree $p$ extension of $F$.
	In particular, we see that $B^\sigma A \mathbf{e}$ is a basis of $M$. This means the connection and Frobenius 
	matrices for $B^\sigma A \mathbf{e}$ have entries in $\mathcal{E}_{F,L}$. Finally,
	we observe that $B^\sigma A \mathbf{e}$
	is obtained from $B\mathbf{e}$ by multiplication by $A_0^{-1}$, which is contained in 
	$\mathcal{E}_{F_0,L}^\dagger$. This means that the connection and Frobenius matrices for 
	$B^\sigma A \mathbf{e}$ have overconvergent entries.
	
\end{proof}

\begin{proof} (Of Theorem \ref{theorem: Berthelot-Kedlaya}) 
	First assume that $X$ is projective over $F$. Take an embedding $X \hookrightarrow \mathbb{P}_{F}^n$ and let
	$\overline{X}$ be the Zariski closure of $X$ in $\mathbb{P}_{\mathfrak{K}[[T]]}^n$. By \cite[Theorem 6.5]{deJong-alterations},
	there exists a finite extension $\Spec(\mathfrak{K}'[[u]]) \to \Spec(\mathfrak{K}[[T]])$, an alteration $\phi_1: X_1 \to X$
	where $X_1$ is a $\Spec(\mathfrak{K}'[[u]])$-variety,
	and an open immersion $j_1: X_1 \to \overline{X}_1$ such that the pair $(\overline{X}_1, \overline{X}_1\backslash X_1)$ is semi-stable. Furthermore, we may assume that $\overline{X}_1\backslash X_1$ is equal to the special fiber. Let $F_1= \mathfrak{K}'((u))$ and consider the map
	$\psi_1:\overline{X}_1 \times_{\Spec(\mathfrak{K}'[[u]])} \Spec(F_1) \to \Spec(F_1)$. By Lemma \ref{lemma: log-crystals}
	we know that $R^i(\psi_1)_*(\mathcal{O}_{X_1 \times_{\Spec(\mathfrak{K}'[[u]])} \Spec(F_1),cris})$
	is an object of $\Fisoc^\dagger(\Spec(F_1))$. Then by applying Corollary \ref{corollary: OC using projector} to the
	map of $\Spec(F_1)$-varieties $X_1 \times_{\Spec(\mathfrak{K}'[[u]])} \Spec(F_1) \to X \times_{\Spec(F)}\Spec(F_1)$ induced by $\phi$, we see that 
	$R^i (\phi\times \Spec(F_1))_* \mathcal{O}_{X \times_{\Spec(F)}\Spec(F_1),cris}$ is an object of
	$\Fisoc^\dagger(\Spec(F_1))$. By base change for crystalline cohomology (see e.g. \cite[Corollary 7.12]{Berthelot-Ogus}) and Lemma \ref{lemma: finite field extension and OC} this implies $R^i \phi_*\mathcal{O}_{X,cris}$ is overconvergent. 
	
 	Now consider general $V$. By \cite[Theorem 4.1]{deJong-alterations} we know that there is an alteration
 	$V_1 \to V$ that is projective and that $V_1$ will be smooth over a finite extension of $F$. The result follows by
 	repeating an argument from the previous paragraph. 
\end{proof}

\bibliographystyle{plain}
\bibliography{finalversion.bib}

\end{document}